\documentclass{siamltex}

\usepackage{amsfonts}
\usepackage{amssymb,enumitem}
\usepackage{algorithm}
\usepackage{algorithmic}
\usepackage{subfigure}

\usepackage{color,xcolor}
\usepackage{pgfplots,filecontents,tikz,multirow,multicol}
\usepackage{graphicx,float}
\usepackage{hyperref}
\usepackage{amsmath,cases}
\usepackage{verbatim}

\def\norm#1{\|#1\|}

\usepackage{array}
\newcolumntype{P}[1]{>{\centering\arraybackslash}p{#1}}
\newcolumntype{M}[1]{>{\centering\arraybackslash}m{#1}}

\RequirePackage[normalem]{ulem} 
\providecommand{\DIFdeltex}[1]{{\protect\color{red}\sout{#1}}}                      
\usepackage{xcolor}
\usepackage{changebar}
\newif\ifdiff
\difftrue

\ifdiff
  \newcommand{\del}[1]{\DIFdeltex{#1}}

\else
  \newcommand{\del}[1]{}

\fi

\newcommand{\tr}[1]{\textcolor{red}{#1}}
\def\argmin{\mathop{\rm argmin}}

\def\diag{\mathrm{diag}}
\def\rank{\mathrm{rank}}

\def\vu{{\bf u}}
\def\vv{{\bf v}}
\def\vw{{\bf w}}

\def\cA{{\cal A}}

\def\cS{{\cal S}}

\def\R{\mathbb{R}}

\def\mE{\mathbb{E}}

\def \tr{\textcolor{black}}

\usepackage{subfigure}

\renewcommand{\figurename}{{\textbf{Figure}}}
\renewcommand{\tablename}{\textbf{Table}}

\title{{Selecting Regularization Parameters for Nuclear Norm Type Minimization Problems}\thanks{Submitted to the editors \today.  This work is supported by NSFC Grant Nos. 11871210, 11971215 and 61971292; HKRGC Grants Nos. CUHK14301718, CityU11301120, and C1013-21GF; and CityU Grant 9380101.}}

\author{Kexin Li\thanks{Key Laboratory of Computing and Stochastic Mathematics (LCSM), School of Mathematics and Statistics, Hunan Normal University, Changsha, Hunan, China.  (Corresponding author: Y. Wen  Email: wenyouwei@gmail.com)}
\and Hongwei Li\thanks{Beijing Advanced Innovation Center for Imaging Theory and Technology Capital Normal University, Beijing, China. }
\and Raymond H. Chan\thanks{Department of Mathematics, City University of Hong Kong, Tat Chee Avenue, Kowloon Tong, Hong Kong SAR, China; Hong Kong Centre for Cerebro-Cardiovascular Health Engineering. }\and You-wei Wen\footnotemark[2] }

\begin{document}
\maketitle

\begin{abstract}
The reconstruction of low-rank matrix from its noisy observation finds its usage in many applications. It can be reformulated into a constrained nuclear norm minimization problem, where the bound $\eta$ of the constraint is explicitly given or can be estimated by the probability  distribution of the noise.
When the Lagrangian method is applied to find the minimizer, the solution can be obtained by the singular value thresholding operator where the
thresholding parameter $\lambda$ is related to the Lagrangian multiplier. In this paper, we first show that the Frobenius norm of the discrepancy between the minimizer and the observed matrix is a strictly  increasing function of $\lambda$. From that we derive a closed-form solution for $\lambda$ in terms of $\eta$. The result can be used to solve the constrained nuclear-norm-type minimization problem when $\eta$ is given. For the unconstrained nuclear-norm-type regularized problems, our result allows us to automatically choose a suitable regularization parameter by using the discrepancy principle. The regularization parameters obtained are comparable to (and sometimes better than) those obtained by Stein's unbiased risk estimator (SURE) approach while the cost of solving the minimization problem can be reduced by {11--18 times}. Numerical experiments with both synthetic data  and real MRI data are performed to validate the proposed approach.
\end{abstract}


\section{Introduction}
In the past {two decades}, the low-rank matrix recovery problem has attracted much attention in computer vision, pattern recognition, image processing, machine learning, and optimization communities.
The problem arises in many applications such as denoising \cite{CandesUnbiased, Gavish2017opt}, inpainting \cite{dong2013nonlocal, Ma2011Fixed},
matrix completion \cite{2016robust, Candes2009exact}, background subtraction \cite{oh2015fast, wright2009robust}, and image alignment \cite{peng2012rasl:}.
In these applications, we often need to solve the constrained minimization problem
\begin{equation}\label{constrained}
\min_{\norm{X-Y}_F^2\leq \eta}f(X)
\end{equation}
to recover an unknown low-rank matrix $X\in \R^{m \times n}$ from its noisy observation $Y$, which is corrupted by Gaussian white noise.
Here $\|\cdot\|_{F}$ denotes the Frobenius norm defined as the square root of the sum of the squares of its elements. Without loss of  generality, we assume that $m\leq n$. The function $f(X)$ is a regularization function of $X$ which incorporates the prior information on the cleaned matrix $X$, and it  can be the rank of $X$, the nuclear norm (NN) \cite{CaiA, Candes2009exact}, the truncated nuclear norm (TNN) \cite{hu2012fast,lee2016computationally, liu2015truncated} or the generalized weighted nuclear norm (GWNN) \cite{Gu2014WNNM, Yan2019Low} of the matrix $X$, etc.
{When $\norm{Y}_F^2\leq \eta$, we have a trivial solution $X=\mathbf 0$ as $f(\mathbf 0)=0$. In order to exclude this case, we assume that
\begin{equation}\label{normYbound}
\norm{Y}_F^2> \eta >0
\end{equation}
throughout this paper and we look for non-trivial solutions of \eqref{constrained}.}

Generally, a constrained minimization problem can be solved by
projected gradient descent method \cite{1997Nonlinear, 2003Con},
Lagrangian multiplier method \cite{1997Nonlinear} and so on.
When the projected gradient descent method is applied, the objective function $f(X)$ should be differentiable.
However, the functions $f(X)$  considered in this paper are non-differentiable, so the projected gradient descent method cannot be directly applied to find the minimizers.

The Lagrangian multiplier method is also widely applied to solve the constrained minimization problem, which converts the constrained minimization problem \eqref{constrained} into a regularized one as follows
\begin{equation}\label{unconstrain}
\min_X f(X)+\frac{1}{2\lambda}\norm{X-Y}_F^2,  
\end{equation}
where ${1}/{\lambda}$ is the Lagrangian multiplier associated with the inequality constraint.
In regularization theory, $\lambda$ is also called the regularization parameter. If the singular value decomposition {(SVD)} of $Y$ is given, the optimal solution of \eqref{unconstrain} can be obtained by soft thresholding of $Y$'s singular values when the regularization function is the nuclear norm \cite{CaiA, Candes2009exact}. The {threshold value} equals to the regularization parameter $\lambda$, see Theorem \ref{th:SVT} below or \cite{CaiA}. By an abuse of notation, we use $\lambda$ to refer to both the threshold value and the regularization parameter.
It is a crucial issue to choose $\lambda$. If $\lambda$ is too small, then the shrinkage is insufficient and the corresponding solution is still noisy; while if $\lambda$ is too large, informative structures might be removed together with the noise
and the solution is a poor approximation of the cleaned matrix.

Ideally, the optimal $\lambda$ should minimize the {mean-squared error (MSE)} between the true matrix and the estimated matrix.
It is difficult to implement this approach because the true matrix is unknown.
Some literatures have discussed how to choose an optimal regularization parameter for the unconstrained minimization problem \eqref{unconstrain}.
In \cite{gavish2014the}, Gavish and Donoho studied the asymptotic MSE (AMSE) and proposed the AMSE-optimal choice of  {hard threshold}, which is $(4/\sqrt{3})\sqrt{n}\tau $ when $m=n$ and {the standard deviation  $\tau$ of the noise} is known.
In \cite{Gavish2017opt}, Gavish {\it et al.} applied the asymptotic framework to find the optimal  {threshold value}, either analytically or numerically, for a variety of loss functions, including Frobenius norm, the nuclear norm  and the operator norm.
In \cite{yadav2017a}, Yadav {\it et al.} applied random matrix theory to infer the AMSE without the knowledge of {the true matrix} $X$.
In \cite{CandesUnbiased}, Cand\'es {\it et al.} applied Stein's unbiased risk estimator (SURE), an unbiased estimate of MSE, to choose the  {soft threshold}.
They also gave an expression for the divergence of the estimated matrix with respect to the observation matrix.
In \cite{bigot2016generalized}, Bigot {\it et al.} derived the corresponding generalized SURE (GSURE) formula for different noise types, which further realized the adaptive thresholding of singular values.
Deledalle {\it et al.} \cite{2014stein}  proposed the Stein unbiased gradient estimator of the risk (SUGAR) {and} provided an asymptotically unbiased estimate of the gradient of the risk,  which is an effective strategy to automatically optimize a collection of parameters.

In fact, the SURE approach requires the knowledge of the noise variance. When the knowledge of the noise variance is unknown,
the generalized cross validation (GCV) method can be applied to {choose} the regularization parameter \cite{golub1979generalized}.
In \cite{josse2016adaptive}, Josse {\it et al.} integrated the SURE method and the GCV method to handle the case when the variance is unknown.
{Furthermore, Hansen {\it et al.} \cite{1993The} proposed a method for choosing the regularization parameter based on the L-curve, which is a log-log plot of the norm of a regularized solution versus the norm of the corresponding residual norm. The optimal regularization parameter is the corner point of the L-curve, see \cite{1994Using, hansen1992analysis}. }

The discrepancy principle is also widely applied to choose the regularization parameter \cite{Stephan2010Morozov, Morozov1984, 1999Applications, 2012Parameter}.
It states that the regularization parameter $\lambda$ should be chosen such that the minimizer $X(\lambda)$ of \eqref{unconstrain} has a discrepancy that equals to the  bound $\eta$ if  $\eta$ is explicitly given, that is
\begin{equation}\label{rootalpha}
\norm{X(\lambda)-Y}_F^2=\eta.
\end{equation}
Obviously, \eqref{rootalpha} is the complementary condition associated with the constrained minimization problem \eqref{constrained},
which shows that there is a mapping between the parameter $\lambda>0$ and the bound $\eta$.
Given a regularization parameter $\lambda$, it is trivial to obtain the bound $\eta$ through directly calculating $X(\lambda)$ and the squared Frobenius norm of $X(\lambda)-Y$.
It is a forward problem.
However, given the bound $\eta$ it is not trivial to obtain the corresponding regularization parameter $\lambda$ because it is the root of the nonlinear equation (\ref{rootalpha}) and it is an inverse problem.
{Existing work on the discrepancy principle  is mainly for  Tikhonov \cite{Stephan2010Morozov, bonesky2008} and TV \cite{2012Parameter} regularization  terms. To the best of our knowledge, there is no literature on considering how to obtain  $\lambda$ for a given $\eta$ in the nuclear-norm-type regularization minimization problems.}

In this paper, we aim precisely to obtain $\lambda$ for any given $\eta$ in \eqref{rootalpha} for nuclear-norm type minimization problems. We first {show} that the norm of the residual
$\norm{X(\lambda)-Y}_F^2$ is a  strictly increasing  function with respect to $\lambda$. Hence we derive a closed form formula for $\lambda$ associated with the constraint \eqref{rootalpha}.
This result can be applied to select an optimal regularization parameter $\lambda$ for the unconstrained minimization problem \eqref{unconstrain}.
Compared with the SURE methods or the GCV methods, the proposed method can directly compute $\lambda$ without solving a complicated and intricate optimization problem. We will see that the cost of solving the unconstrained minimization problem \eqref{unconstrain} (and hence the constrained minimization problem \eqref{constrained} too) is reduced by {11} times when compared with SURE methods. We remark that our approach also provides a fast method for solving the constrained minimization problem \eqref{constrained} for any given $\eta$ by solving the corresponding unconstrained problem \eqref{unconstrain} with the corresponding $\lambda$.

Note that the low-rank matrix recovery problem usually involves finding the SVD  of matrices, and its computational complexity  is $O(mn\min\{m,n\})$. For large-scale problems, the cost of the traditional SVD is very expensive.
In the past ten years, randomized algorithms have been more and more widely used in low-rank matrix approximations \cite{2018faster, 2011A, oh2015fast, 2013regularization}.
Compared with classic algorithms, randomized algorithms involve fewer floating-point operations (flops) and are more effective for large-scale problems. Our strategy for choosing the  {regularization value} can also be easily incorporated into the randomized algorithms and the cost can be reduced by 18 times when compared with SURE methods.

{The outline of this paper is organized as follows.
	In Section \ref{Pre}, we consider the solution of the regularization model \eqref{unconstrain} when the penalty function is the nuclear norm or its variants.
	In  Section \ref{conreg}, we discuss the relationship between the regularized model and the constrained model. The estimation of the bound $\eta$ for the residual is also given.
	In Section \ref{Reg}, we solve \eqref{rootalpha} for the {regularization parameter} $\lambda$ in terms of the bound $\eta$. Then we consider the solution of the constrained model \eqref{constrained} for various nuclear-norm-type problems. {In  Section \ref{numerical}, numerical results are given to demonstrate the effectiveness of the proposed method.}
	Finally, a short conclusion is given in  Section \ref{Conclusion}.

\section{Regularized Minimization Model}\label{Pre}
{In this section, we introduce some basic notation and consider the solution for the regularization minimization model in \eqref{unconstrain}.}

\subsection{ Notation}
We first introduce the singular value decomposition (SVD).
Any $m\times n$ $(m\leq n)$ matrix $Y$  can be factorized as follows
$Y = U\Sigma V^T$
where $U=[\vu_1,\vu_2,\ldots,\vu_m]\in\R^{m\times m}$ and $V=[\vv_1,\vv_2,\ldots,\vv_m]\in{\R^{n\times m}}$
are orthogonal, and $\Sigma \in \R^{m\times m} $  is a diagonal matrix.
The column vectors $\vu_i$  and $\vv_i$ are  respectively called the $i$-th left and right singular vectors,
and the diagonal entry $\Sigma_{i,i} = \sigma_{i}$ is called the $i$-th singular value of the matrix $Y$. We assume
$\sigma_1 \geq \sigma_2 \geq ... \geq \sigma_m\geq 0$.

The SVD can also be reformulated by the outer product form  $Y=\sum_{i=1}^m\sigma_i\vu_i\vv_i^T$.
A $k$-truncated matrix of $Y$ is defined as
$Y_k\equiv\sum_{i=1}^k\sigma_{i}\vu_i\vv_i^T. $
By applying Eckart-Young-Mirsky theorem \cite{1936The, 1960Symmetric}, we know that
$$
Y_k=\argmin_{\rank(X)\leq k}\norm{X-Y}_F^2.
$$
We remark that the $k$-truncated matrix of $Y$ is not unique  if $\sigma_{k+1}=\sigma_{k}$.

The main aim in \eqref{constrained} is to recover an unknown matrix $X$ from the noisy observation $Y$.
One common assumption is that the matrix $X$ has low-rank structure.
It is natural to choose the objective function $f(X)=\rank(X)$ and consider the following problem
\begin{equation}\label{LRMA3}
\min_X  \rank(X)+\frac{1}{2\lambda}\norm{X-Y}_F^2.
\end{equation}
The rank of a matrix can be defined as the nonzero numbers of its singular values, i.e.,  the $\ell_0$-norm  of the singular values.
The solution to (\ref{LRMA3}) is given by the hard thresholding operator $X=\sum_{i=1}^m\mathrm{HT}_\lambda(\sigma_{Y,i})\vu_i\vv_i^T$,
where $\mathrm{HT}_\lambda(x)$ is a map defined by
\[
\mathrm{HT}_\lambda(x)=
\left\{
\begin{array}{ll}
x,& x>\sqrt{2\lambda},\\
0 \;\mathrm{or}\; x, & x=\sqrt{2\lambda},\\
0, & x<\sqrt{2\lambda}.
\end{array}
\right.
\]

\subsection{Nuclear Norm}
In compressed sensing, in order to characterize the sparsity of the signal, the $\ell_1$-norm is usually used to approximate the $\ell_0$-norm.
Similarly, in the rank minimization problem (\ref{LRMA3}), the rank of a matrix can be approximately replaced by the $\ell_1$-norm of the singular values, which is the nuclear norm of the matrix.
\begin{definition}
Given a matrix $X\in \R^{m\times n}$ with $\sigma_{X,i}$ being its $i$-th singular values,
the nuclear norm (NN) $\norm{X}_*$ of $X$ is defined as the sum of its singular values, i.e., $\norm{X}_*\equiv\sum_{i=1}^m \sigma_{X,i}$.
\end{definition}

The nuclear norm regularized minimization problem is given by
\begin{equation}\label{NNM}
\min_X \|X\|_*+\frac{1}{2\lambda}\norm{X-Y}_F^2.
\end{equation}
The following theorem states that the optimal minimizer of \eqref{NNM}
is a function with respect to the regularization parameter $\lambda$.

\begin{theorem}\label{th:SVT} \cite{CaiA}
Let $Y$ be an $m\times n$ matrix,  $\vu_i$ and $\vv_i$ be the $i$-th left and right singular vectors of $Y$, and
$\sigma_{Y,i}$ be the $i$-th singular value.
Then the solution of \eqref{NNM} is given by $\widehat{X}(\lambda)= \operatorname{SVT}_\lambda(Y)$, where $\operatorname{SVT}_\lambda(Y)$ is the singular value thresholding ({SVT}) operator defined as
\begin{equation}\label{svt}
\widehat{X}(\lambda)={\operatorname{SVT}}_\lambda(Y)\equiv\sum_{i=1}^{m}(\sigma_{Y,i}-\lambda)_+ \vu_i\vv_i^T.
\end{equation}
{Here $(x)_+\equiv \max(x,0)$.}
\end{theorem}

\subsection{Truncated Nuclear Norm}\label{subsectionTNN}
The nuclear norm shrinks all singular values equally and the large singular values are usually over-penalized.
To overcome the shortcomings of nuclear norm,
Hu {\it {\it et al.}} \cite{hu2012fast} used the truncated nuclear norm  instead of the standard nuclear norm in order to keep the largest $r$ singular values unchanged while shrinking the other singular values.
\begin{definition}
Given a matrix $X\in \R^{m\times n}$ with {$\sigma_{X,i}$} being its $i$-th singular values,
the truncated nuclear norm (TNN) $\norm{X}_r$ of $X$ is defined as the partial sum  of its singular values, i.e., $\norm{X}_r\equiv\sum_{i=r+1}^m \sigma_{X,i}$,  where $r<m$ is an integer.
\end{definition}

The TNN regularized minimization problem is given by
\begin{equation}\label{tnnm}
\min _{X} \norm{X}_r+ \frac{1}{2\lambda}\norm{Y-X}_{F}^{2}.
\end{equation}
The TNN has been proposed in low-rank {matrix} recovery and matrix completion problem \cite{hu2012fast,lee2016computationally, liu2015truncated}.
The TNN is non-convex and the solution of TNN regularized minimization problem is generally non-trivial.
In  \cite{hu2012fast}, the Von Neumann's trace inequality \cite{1975Mirsky} was  utilized to handle the non-convexity. The idea is to rewrite TNN into an equivalent form
$$\norm{X}_r=\norm{X}_*- \min_{(U,V)\in \cA}\mathrm{trace}(UXV^T)$$
where $\cA=\{(U,V): UU^T=I,VV^T=I, U\in\R^{r\times m}, V\in\R^{r\times n}\}$.

According to Von Neumann's lemma,  the inner product of two matrices is always bounded by the sum of the products of their corresponding singular values,
while the maximum of the inner product of two matrices can only be achieved when these two matrices have the same left and right singular vectors.
Based on these facts,   the optimal solution of \eqref{tnnm}  can be expressed by the partial singular value thresholding operator \cite{2016Partial}.

\begin{theorem} 
For any $\lambda>0$, {a global solution} of \eqref{tnnm} with a target rank $r<m$ is given by the partial singular value thresholding ({PSVT})
\begin{equation}\label{PSVT}
\operatorname{PSVT}_{r,\lambda}(Y)\equiv\sum_{i=1}^{r} \sigma_{Y,i} \vu_i\vv_i^T+ \sum_{i=r+1}^{m} (\sigma_{Y,i}-\lambda)_+ \vu_i\vv_i^T.
\end{equation}
\end{theorem}

We note that the $\operatorname{PSVT}$ operator is comprised of two terms: the first term is the $r$-truncated matrix $Y_r \equiv \sum_{i=1}^r\sigma_{Y,i}\vu_i\vv_i^T$,
and the second term is the singular value thresholding operator of $Z=Y-Y_r$, i.e., ${\operatorname{SVT}}_\lambda(Z)=\sum_{i=r+1}^{m} (\sigma_{Y,i}-\lambda)_+ \vu_i\vv_i^T$.
Hence we have
\[
\operatorname{PSVT}_{r,\lambda}(Y)=Y_r+{\operatorname{SVT}}_\lambda(Y-Y_r),
\]
which implies that the truncated nuclear norm minimization problem can be solved through the nuclear norm minimization problem. For the constrained minimization problem, we have
\begin{equation}\label{tnn2nn}
\argmin_{\norm{Y-X}_F^2\leq \eta} \norm{X}_{r}\equiv Y_r + \argmin_{\norm{Z-X}_F^2\leq \eta} \norm{X}_{*},
\end{equation}
where $Z=Y-Y_r$.
Thus the constrained minimization problem for the truncated nuclear norm can be reformulated as that of the nuclear norm.

\subsection{Weighted Nuclear Norm}\label{wnn}
To improve the flexibility of the nuclear norm, Gu {\it et al.}  \cite{Gu2014WNNM} proposed a weighted nuclear norm to  approximate the rank function, which has good performance in image denoising problem.
\begin{definition}
Given a matrix $X\in \R^{m\times n}$ with $\sigma_{X,i}$ being its $i$-th singular values and the non-negative weight vector $\vw=(w_1,w_2,\ldots,w_m)$, 
the weighted nuclear norm (WNN) $\norm{X}_{\vw,*}$ of $X$ is defined by $\norm{X}_{\vw,*}\equiv\sum_{i=1}^m w_i \sigma_{X,i}$.
\end{definition}

We remark that the standard nuclear norm and the truncated nuclear norm can be regarded as the special version of the weighted nuclear norm, and {the weight vectors are $\vw=(1,\ldots,1)$ for the standard nuclear norm and $\vw=(0,\ldots,0,1,\ldots,1)$ for the truncated nuclear norm.}

We emphasize that the choice of the weights is very important.
A reasonable weight vector should guarantee that the resulting singular values are in a non-increasing order.
When the weights are arranged in non-decreasing order,
larger singular values should be less penalized.
In \cite{Gu2014WNNM}, the weights are suggested to be
$$
w_i=(\sigma_{X,i}+\epsilon)^{-1}, \quad i=1,...,m,
$$
where the parameter $\epsilon$ is a sufficiently small positive number in order to avoid dividing by zero and should be set slightly smaller than the expected nonzero singular value of $X$.
In this case, $\norm{X}_{\vw,*}$  approximates the number of the nonzero singular values of $X$,
i.e., the rank of $X$.
In \cite{huang2018rank}, Huang {\it et al.} generalized the weights to
\begin{equation}\label{pnorm}
w_i=(\sigma_{X,i}+\epsilon)^{p-1}, \quad  0\leq p<1.
\end{equation}
In this way, $\norm{X}_{\vw,*}$  approximates  the Schatten $p$-norm of $X$, which is defined by $ (\sum_{i=1}^m \sigma_{X,i}^p  )^{1/p}$.
Because the singular values  $\sigma_{X,i}$ are unknown,
an iterative approach is applied to estimate them.
It is proved in \cite{huang2018rank} that the optimal solution of the reweighted nuclear norm minimization problem is also a solution of the Schatten $p$-norm minimization problem.

In this paper, the weight vector defined in \cite{huang2018rank} ({i.e.,} {\eqref{pnorm}}) is adopted,
and we call the corresponding norm the generalized weighted nuclear norm (GWNN).  Although the weighted nuclear norm is non-convex, its global optimal solution can still be obtained.
Similar to Theorem \ref{th:SVT}, the closed form solution for the reweighted problem \eqref{WNNM} can be represented by weighted singular value thresholding.
\begin{theorem} \cite{gu2017weighted}
For any $\lambda>0$, $Y\in \R^{m\times n},$ if the weights $\{w_i\}_{i=1}^m$ satisfy $0\leq  w_1\leq w_2\leq...\leq w_m $, then {a global solution} of the following minimization problem
\begin{equation}\label{WNNM}
\min _{X} \norm{X}_{{\vw},*} + \frac{1}{{2\lambda}}\norm{Y-X}_{F}^{2}
\end{equation}
is given by the weighted singular value thresholding ({WSVT}) operator
\begin{equation}\label{WSVT}
\widehat{X}_{{\vw}}(\lambda)=\operatorname{WSVT}_\lambda(Y)\equiv\sum_{i=1}^{m}(\sigma_{Y,i}-\lambda w_i)_+ \vu_i\vv_i^T.
\end{equation}
\end{theorem}

If there are zero entries in the weight vector $\vw$, i.e., $\vw= (0,...,0,w_{r+1},...,w_m)$ with $0<w_{r+1}\leq w_{r+2}\leq...\leq w_m$, we have
\begin{equation}\label{nnwnn2pwnn}
\operatorname{WSVT}_\lambda(Y)=Y_r+\sum_{i=r+1}^{m}(\sigma_{Y,i}-\lambda w_i)_+ \vu_i\vv_i^T.
\end{equation}
This is because the WSVT operator is comprised of the $r$-truncated matrix and the WSVT operator with nonzero weights.
Therefore, without loss of generalization, in the following when we talk about weighted nuclear norm, we only focus on the case where all the weights are positive.

%

\section{Constrained model and regularized model}\label{conreg}
Because the minimization problem of the truncated nuclear norm or weighted nuclear norm with zero weights can be reformulated as that of the nuclear norm or the weighted nuclear norm with positive weights respectively ({see \eqref{tnn2nn} or \eqref{nnwnn2pwnn}}),
in this section, we assume that the objective function $f(X)$ is either the nuclear norm or weighted nuclear norm with positive weights only.

We discuss the relationship between the constrained model \eqref{constrained} and the regularized  model \eqref{unconstrain}.
Mathematically, the two minimization models are equivalent in the sense that
given a parameter $\lambda$, there exists a corresponding $\eta$ such that the solution of \eqref{unconstrain} is also a solution of \eqref{constrained}, and vice versa.
The following lemmas give the detailed explanations.

\begin{lemma}\label{lemma1}
	If $X^\dagger$ is a {global minimizer} of the unconstrained regularized  problem \eqref{unconstrain} with $\lambda>0$,
	then there exists an $\eta\geq 0$ such that $X^\dagger$ is also a {global minimizer} of \eqref{constrained}.
\end{lemma}
\begin{proof}
	Define $\eta\equiv\|Y- X^\dagger\|_F^2 \ge 0$.
	For any $X$ satisfying $\|Y-X\|_F^2\leq \eta$, applying the fact that $X^\dagger$ is a
	{global minimizer} of \eqref{unconstrain}, we have
	\[
	\frac{\eta}{2\lambda}+  f(X^\dagger)=\frac{1}{2 \lambda}\|Y- X^\dagger\|_F^2+ f(X^\dagger)\leq \frac{1}{2\lambda}\|Y-X \|_F^2+ f(X)\leq \frac{\eta}{2\lambda} +  f(X).
	\]
	Hence $f(X^\dagger) \leq f(X)$,
	which implies that $X^\dagger$ is also a {global} minimizer of \eqref{constrained}.
\end{proof}

In order to exclude a trivial solution $X=0$ for the constrained minimization problem in \eqref{constrained}, we have assumed that $\norm{Y}_F^2 > \eta$, see \eqref{normYbound}.
We show that a minimizer of the constrained minimization problem  \eqref{constrained} always lies on  the boundary of the constraint if $\norm{Y}_F^2 > \eta$.
\begin{lemma}\label{solboundary}
If $X^\dagger$ is a {global minimizer} of \eqref{constrained} with $\norm{Y}_F^2 > \eta > 0$, then we have
$\left\|Y-X^{\dagger}\right\|_{F}^{2}=\eta$.
\end{lemma}

\begin{proof}
Let  $\sigma_{X^\dagger,i}$ and $\sigma_{Y,i}$ be the singular values of $X^\dagger$ and $Y$ respectively.
If, by contradiction, $X^{\dagger}$ is an interior point of the constrained set,
then we have
\[
\sum_{i=1}^m  (\sigma_{Y,i}-\sigma_{X^{\dagger},i})^2\leq
\norm{Y-X^\dagger}_F^2< \eta,
\]
where the first inequality is obtained by applying Von Neumann's trace inequality \cite{1975Mirsky}.
Let $\sigma_{X^\dagger,k}$ be the smallest positive singular value of $X$.
The above inequality implies that there exists  a positive scalar $\delta$ with $0<\delta<\sigma_{X^\dagger,k}$ such that
\begin{equation}\label{aaa}
\sum_{i=1}^{k-1}\left(\sigma_{Y,i}-\sigma_{X^\dagger,i} \right)^2+\left(\sigma_{Y,k}-\sigma_{X^\dagger,k}+\delta \right)^2+\sum_{i=k+1}^{m}\left(\sigma_{Y,i} \right)^2\leq\eta.
\end{equation}
Let $\Sigma_{\widehat{X}}$ be the diagonal matrix with the diagonal entries $\sigma_{\widehat{X},k}=\sigma_{X^\dagger,k}-\delta$ and $\sigma_{\widehat{X},i}=\sigma_{X^\dagger,i}$ for $i \neq k$.
Clearly, the diagonal entries of $\Sigma_{\widehat{X}}$ still keep the non-increasing ordering.

{Let the SVD of $Y$ be $Y=U \Sigma_Y V^T$ and we set} $\widehat{X}=U\Sigma_{\widehat{X}} V^T$.
By (\ref{aaa}), $\norm{Y-\widehat{X}}_F^2< \eta$, i.e., $\widehat{X}$ is a feasible solution of \eqref{constrained}. However, we have
\[
\sum_{i=1}^m w_i\sigma_{\widehat{X},i}=f(\widehat{X})<f(X^\dagger)=
\sum_{i=1}^m w_i\sigma_{X^\dagger,i},
\]
(for nuclear norm, all $w_i = 1$).
This contradicts the assumption that $X^\dagger$ is a global minimizer
of \eqref{constrained}.
Therefore, we must have $\left\|Y-X^{\dagger}\right\|_{F}^{2}=\eta$.
\end{proof}

\begin{lemma}\label{lemma2}
	Let $\norm{Y}_F^2 > \eta > 0$ and the objective function $f(X)$ in \eqref{constrained} be the nuclear norm or the weighted nuclear norm with weights $0 \tr{<}  w_1\leq w_2\leq...\leq w_m $. If $X^\dagger$ is a global minimizer of \eqref{constrained}, then there exists a parameter $\lambda> 0$ such that $X^\dagger$ is also a {global minimizer} of \eqref{unconstrain}.
\end{lemma}
\begin{proof}
First, we show that
$X^\dagger$ and $Y$ have the same sets of left and right singular vectors.
For if otherwise, then by Von Neumann's trace inequality \cite{1975Mirsky}, we have
\[
\sum_{i=1}^m  (\sigma_{Y,i}-\sigma_{X^{\dagger},i})^2<
\norm{Y-X^\dagger}_F^2=\eta,
\]
where the last equality is by Lemma \ref{solboundary}.
Similar to the proof in Lemma \ref{solboundary}, we can then construct a matrix $\widehat{X}$ such that $f(\widehat{X})<f(X^\dagger)$ \tr{while
$\norm{Y-\widehat{X}}_F^2\leq\eta$.} This contradicts the assumption that $X^\dagger$ is a global minimizer of \eqref{constrained}.

Next, we show that the singular values  of $X^\dagger$ can be obtained by the Lagrangian multiplier method.
Let $Y=U\Sigma_YV^T$ be the SVD of $Y$. By the above paragraph, the solution of \eqref{constrained} is given by $X^\dagger=U\Sigma_{X^\dagger} V^T$. Here $\Sigma_{X^\dagger}$ can be obtained by solving the minimization problem
\begin{equation}\label{temp1}
\min_{\norm{\Sigma_{X}-\Sigma_{Y}}_F^2\leq \eta}f(\Sigma_{X}).
\end{equation}
\tr{We note that the constraint $\norm{\Sigma_{X}-\Sigma_{Y}}_F^2\leq \eta$ must be an active constraint. For if not, then $X = 0$ is a minimizer and we obtain $\norm{\Sigma_{Y}}_F^2=\norm{Y}_F^2\leq \eta$, a contradiction.}

Without loss of generality, we consider $f(\Sigma_{X})=\sum_iw_i\sigma_{X,i}$ (for nuclear norm, all $w_i=1$).
Then (\ref{temp1}) can be reformulated into
\begin{equation} \label{temp3}
    \min_{\sigma_{X}\in \cS}\sum_iw_i\sigma_{X,i},
\end{equation}
where
$$
\cS=\{\sigma_{X}=(\sigma_{X,1},\ldots,\sigma_{X,m})^T: \norm{\sigma_{X}-\sigma_{Y}}_2^2\leq \eta, \sigma_{X,i}\geq \sigma_{X,2}\geq \cdots \geq \sigma_{X,m} \geq 0 \}.
$$
Since $\sigma_Y\in \cS$ is a strictly feasible point, the minimization problem satisfies Slater's condition \cite[Proposition 3.3.9]{1997Nonlinear}.
Therefore, there must exist a Lagrangian multiplier $\alpha>  0$ such that the solution of (\ref{temp3}) \tr{(which is the same as the solution of \eqref{constrained})} can be obtained by minimizing the Lagrangian function
\[
\min_{\sigma_{X}\geq 0}\sum_{i=1}^m w_i\sigma_{X,i}+\frac{\alpha}{2}\left(\norm{\sigma_{X}-\sigma_{Y}}_2^2-\eta\right).
\]
The objective function is separable in each $i$, hence we obtain
\begin{equation} \label{sigma1}
\sigma_{X^\dagger,i}=(\sigma_{Y,i}-w_i/\alpha)_+ \quad \forall i=1, \ldots, m.
\end{equation}
It is easy to check that we have $\sigma_{X^\dagger,i}\geq \sigma_{X^\dagger,i+1}$ due to the ordering of the singular values $\sigma_{Y,i}$ and the weights $\omega_i$.

\tr{Finally, we show that the lemma holds.
It is obvious that $\sigma_{X^\dagger,i}$  given by (\ref{sigma1}) are precisely those  singular values in the SVT operator (see \eqref{svt}) or a WSVT operator (see \eqref{WSVT}) with the thresholding parameter $\lambda = 1/\alpha$. Therefore, by Theorem \ref{th:SVT} or Theorem \ref{WNNM}, we have
\[
f(X^\dagger)+\frac{\alpha}{2 }\norm{X^\dagger-Y}_F^2 \leq f(X)+\frac{\alpha}{2}\norm{X-Y}_F^2,
\]
for all $X\in \R^{m\times n}$.
}
\end{proof}

Based on the above lemma, we can find a suitable regularization parameter $\lambda$ when the bound $\eta$ is given such that a solution of \eqref{unconstrain} is one of the solution of \eqref{constrained}.

\subsection{Choice of Upper bound} \label{DP}
It is very important to choose a suitable  {regularization parameter} $\lambda$ in the regularized minimization problem \eqref{unconstrain} because
the quality of the recovered matrix highly depends on this parameter.
If $\lambda$ is too large, the given data cannot be fitted correctly,
and if it is too small, the shrinkage is insufficient.
The Morozov discrepancy principle \cite{1996Regularization, Morozov1984} is one method to choose $\lambda$.
This principle selects $\lambda$ such that the residual norm is bounded, i.e., a good regularized solution \tr{$X(\lambda)$} should lie in a set $\{ X: \|Y-X(\lambda)\|_F^2\leq \eta\}$, where $\eta$ is an upper bound of the discrepancy depending on the noise level. Now we consider how to choose the bound $\eta$.

According to the discussion in  Section \ref{Pre},  the minimizer of the regularized problem \eqref{unconstrain} has a {closed form solution} and is a function of the parameter $\lambda$.
We will use $\widehat{X}(\lambda)$ to denote the minimizer.
Notice that the residual $Y-\widehat{X}(\lambda)$ can be rewritten as
$
Y-\widehat{X}(\lambda)={X_{\dagger}}+W -\widehat{X}(\lambda),
$
where {$X_\dagger$ is the true solution and each entry of the error matrix} $W$ is the Gaussian noise with zero mean and variance $\tau^2. $
If $\widehat{X}(\lambda)$ is a good estimation of {$X_\dagger$}, then the residual should be dominated by the Gaussian noise $W$.
We introduce the following theorems which can be found in \cite{ 2010Finding, 2020Mead, 2013chi}.

\begin{theorem}\label{exp}
	Assume that {all the entries} of $W$ are of normal independent distribution with mean 0 and variance $\tau^2$.
	Then $\|Y-X_\dagger\|_F^2$ is $\chi^2$-distributed with variance $\tau^2$ and $mn$ degree of freedom, i.e.,
	$
	\frac{1}{\tau^2}\norm{W}_F^2\sim \chi^2_{mn}.
	$
	Moreover, we have
	$$\mE\norm{W}_F^2 = mn\tau^2.$$
\end{theorem}
\begin{theorem}\label{Upp}
	Suppose that $g$ is a Lipschitz function on matrices:
	$$
	|g({X})-g({Y})| \leq L\|{X}-{Y}\|_{{F}}, \quad \text { for all } {X}, {Y}\in \R^{m\times n}.
	$$
	Given a Gaussian matrix ${W}\sim \mathcal{N}(0,\tau^2 I)$, then
	$$
	\mathbb{P}\{g({W}) \geq \mathbb{E} g({W})+L \tau t\} \leq \mathrm{e}^{-t^{2} / 2}.
	$$
\end{theorem}

Since $\|\cdot\|_F$ is Lipschitz continuous with $L=1$, we obtain
$$
\mathbb{P}\{\|{W}\|_F \leq \mathbb{E} \|{W}\|_F+t\tau \} \geq 1-\mathrm{e}^{-t^{2} / 2}.
$$
We can change $t$ to get an upper bound of the Gaussian noise $\|{W}\|_F$ with high probability (e.g. when $t\ge4$, the noise is bounded with probability greater than 0.9996).

According to Theorems \ref{exp} and \ref{Upp},
it is natural to set the upper bound of $\|W\|_F^2$ to $\eta=mn\tau^2$.
However, $\widehat{X}(\lambda)$ is obtained by applying the SVT operator to $Y$,
hence $\widehat{X}(\lambda)$ is dependent on the noise $W$.
It is reasonable to modify the upper bound to
\begin{equation}\label{upperbound}
	\eta=c mn\tau^2,
\end{equation}
where $c\simeq 1$ can be adjusted appropriately to suit the applications \cite{2010Modular, 1992Methods}. We see that the parameter $\eta$ depends on the estimation of the noise level $\tau$.
If the noise level is not given, $\tau$ can be estimated by using the median rule  \cite{ 1994Ideal, mallat1999wavelet}, i.e.,
\begin{equation}\label{median}
	\tau=\text{median}(|\widehat{Y}_{HH}|)/0.6745,
\end{equation}
where {$\widehat{Y}_{HH}$}  is the high-high coefficients of $Y$ at the finest wavelet transform level.

Once the upper bound $\eta$ is determined, we are in the position to find $\lambda$ in terms of $\eta$.
Because the regularization parameter $\lambda$ can be regarded as the Lagrange multiplier associated with the active constraint,
the complementarity condition states
\begin{equation}\label{MDP}
	\norm{Y-\widehat{X}(\lambda)}_F^2= \eta.
\end{equation}
It means that the optimal regularization parameter $\lambda$ is the root of the nonlinear equation \eqref{MDP}.

\section{Solutions for the Constrained and Unconstrained problems}\label{Reg}
In this section, we first consider finding an analytical expression of  $\lambda$ for the unconstrained problem \eqref{unconstrain} when the bound $\eta$ for the constrained problem \eqref{constrained} is given (e.g. as in \eqref{upperbound}). Since the regularization parameter $\lambda$ is the same as the thresholding parameter used in the singular value soft-thresholding operator, we can hence solve the unconstrained problem \eqref{unconstrain} easily. By the equivalence of the two minimization problems, the constrained problem with the bound $\eta$ is also then solved.

\subsection{Nuclear Norm}\label{nns}
Let $f(X)=\norm{X}_*$, we know that the minimizer of the nuclear norm regularized minimization problem \eqref{NNM} is given by the SVT operator, see Theorem \ref{th:SVT}.
We first show that $\norm{Y-\widehat{X}(\lambda)}_F^2$ is a positive and strictly increasing function of $\lambda$.
\begin{theorem}\label{existence}
	Let $\sigma_{Y,i},i=1,...,m$ be the singular values of $Y\in \R^{m\times n}$ with
	$\sigma_{Y,1}\geq \sigma_{Y,2} {\geq}\ldots \geq \sigma_{Y,m}\geq 0$ and $\widehat{X}(\lambda)$ be the solution of \eqref{NNM} obtained by the SVT operator acting on $Y$.
	Then $\norm{Y-\widehat{X}(\lambda)}_F^2$ is a positive and {strictly increasing} function of $\lambda$ in the interval $[0, \sigma_{Y,1}]$.
\end{theorem}
\begin{proof}
	According to Theorem \ref{th:SVT},
	we have
	\[
	Y-\widehat{X}(\lambda)=\sum_{i=1}^{m}\sigma_{Y,i} \vu_i\vv_i^{\mathrm{T}} - \sum_{i=1}^{m}\max(\sigma_{Y,i}-\lambda,0) \vu_i\vv_i^{\mathrm{T}}=\sum_{i=1}^{m}\min(\sigma_{Y,i},\lambda) \vu_i\vv_i^{\mathrm{T}}.
	\]
	Hence we have
	\begin{equation} \label{psilambda0}
		\phi(\lambda)\equiv\norm{Y-\widehat{X}(\lambda)}_F^2=
		\sum_{i=1}^{m}\min(\sigma_{Y,i}^2,\lambda^2).
	\end{equation}
	Without loss of generality, we assume $\sigma_{Y,1}>0$ and define $\sigma_{Y,m+1}=0$.
	For any $\lambda$ satisfying $0=\sigma_{Y,m+1}< \lambda\leq \sigma_{Y,1}$, there must exist an index $k$ such that
	$\lambda \in (\sigma_{Y,k+1}, \sigma_{Y,k}]$.
	Then we have
	\begin{equation} \label{psilambda}
		\phi(\lambda)= k\lambda^2 + \sum_{i=k+1}^m \sigma_{Y,i}^2,
	\end{equation}
	where the first term is a quadratic function with respect to $\lambda$ and the second {term} is independent of $\lambda$. Thus $\phi(\lambda)$ is a strictly  increasing function in  $(\sigma_{Y,k+1}, \sigma_{Y,k}]$.
	One can easily check that
	\[
	\lim_{\lambda\rightarrow \sigma_{Y,k+1}^-}\phi(\lambda)={\phi}(\sigma_{Y,k+1})=k\sigma_{Y,k+1}^2+ \sum_{i=k+1}^m \sigma_{Y,i}^2,
	\]
	which implies that $\phi(\lambda)$ is a continuous function.
	Thus we can deduce that $\phi(\lambda)$ is a positive and {strictly  increasing} function in the interval $[0, \sigma_{Y,1}]$.
\end{proof}

Based on Theorem \ref{existence}, we know $\phi(\lambda)=0$ for $\lambda=0$ and $\phi(\lambda)=\sum_{i=1}^{m}\sigma_{Y,i}^2=\norm{Y}_F^2$ for $\lambda \geq \sigma_{Y,1}$.
Thus we obtain $0\leq \phi(\lambda)\leq \norm{Y}_F^2$ and we have the following theorems.
\begin{theorem}\label{RegPara}
Let $\widehat{X}(\lambda)$ be the solution of {\eqref{NNM}}  obtained by the SVT operator acting on $Y$ and let $\eta > 0$ be such that $\norm{Y}_F^2> \eta$.
Then the nonlinear equation $\norm{Y-\widehat{X}(\lambda)}_F^2=\eta$ has a {\it unique} solution $\lambda>0$.
\end{theorem}
\begin{proof}
We define a new sequence
\begin{equation}\label{bj}
b_j=j\sigma_{Y,j}^2+\sum_{i=j+1}^{m} \sigma_{Y,i}^2.
\end{equation}
It is obvious that the sequence $b_j$ is non-increasing, {that is} $b_1\geq b_2\geq \ldots\geq b_m$. Define $b_{m+1}=0$.
Since $b_1 > \eta > 0$, we choose the subscript $k$ such that $b_{k+1}<\eta\leq b_k$.
Let
\begin{equation}\label{lambda2}
\displaystyle{\lambda=\sqrt{(\eta-\sum_{i=k+1}^m \sigma_{Y,i}^2)/{k}}.}
\end{equation}
Since $\eta > b_{k+1}$, we have $\lambda > 0$. By \eqref{psilambda}, we have $\phi(\lambda)=\eta$.
\end{proof}

\subsection{Generalized Weighted Nuclear Norm} \label{gwnns}
Now we consider the monotonicity of the discrepancy $\norm{Y-\widehat{X}_{{\vw}}(\lambda)}_F^2$ with respective to the parameter $\lambda$.
For any $0<\lambda\leq \sigma_{Y,1}$, there exists an index $k$ such that $\lambda\in (\sigma_{Y,k+1},\sigma_{Y,k}{]}$. Suppose
$\widehat{X}_{{\vw}}(\lambda)$ is the solution of \eqref{WNNM} with non-descending weights $\{w_i\},i=1,...,m$, then using \eqref{WSVT} and after some manipulations similar to those in Theorem \ref{existence}, we have
\begin{equation} \label{psifun}	{\psi(\lambda)}\equiv\|Y-\widehat{X}_{{\vw}}(\lambda)\|_F^2=\left(\sum_{i=1}^k w_i^2\right)\lambda^2 +\sum_{i=k+1}^m \sigma_{Y,i}^2.
\end{equation}
{Comparing to the function $\phi(\lambda)$ in \eqref{psilambda}, the coefficient $k$ of $\lambda^2$ is replaced by the sum of the weighted square $\sum_{i=1}^k w_i^2$ for the function $\psi(\lambda)$.}
Obviously, we can obtain the following result similar to  Theorem \ref{existence}.
\begin{theorem}\label{existWNNM}
	Let $\widehat{X}_{{\vw}}(\lambda)$ be the solution of \eqref{WNNM} obtained by applying the WSVT operator \eqref{WSVT} acting on $Y$.
	Then the function {$\psi(\lambda)$} is a positive and {strictly  increasing} function of $\lambda$ in the interval $[0, \sigma_{Y,1}]$.
\end{theorem}


{The result for Theorem \ref{RegPara} can be easily extend to the case for the weighted nuclear norm minimization problem.
{Notice that the WSVT operator in \eqref{WSVT} can be reformulated into
$$
\operatorname{WSVT}_\lambda(Y)=\sum_{i=1}^{m}w_i
\left(\frac{\sigma_{Y,i}}{w_i}-\lambda \right)_+ \vu_i\vv_i^T.
$$
Hence the weighted singular values {${\sigma_{Y,i}}/{w_i}$} which are less than the threshold parameter $\lambda$ are set to zero.
Similar to \eqref{bj}, we define the sequence
\begin{equation}\label{wbj}
b_j=(\frac{\sigma_{Y,j}}{w_j})^2\sum_{i=1}^{j}w_i^2+\sum_{i=j+1}^{m} \sigma_{Y,j}^2, \quad j=1,...,m,
\end{equation}
{with $b_{m+1}=0$.} Then $\{b_j\}_{j=1}^{m+1}$ is a non-increasing sequence.
Since by  Theorem \ref{existWNNM}, the function $\psi(\lambda)$ defined by \eqref{psifun} {is} a positive and strictly increasing function of the parameter $\lambda$, we see that there exists an index $k$ such that $b_{k+1}<\eta\leq b_k$.
Hence we have the following theorem.
}
\begin{theorem}\label{RegparaWNNM}
Let $\widehat{X}_{{\vw}}(\lambda)$ be the solution of \eqref{WNNM} obtained by applying the WSVT operator \eqref{WSVT} on $Y$.
Then the nonlinear equation $\psi(\lambda)\equiv \norm{Y-\widehat{X}_{{\vw}}(\lambda)}_F^2=\eta$ has a {\it unique} solution
\begin{equation}\label{wlambda}
\lambda=\sqrt{(\eta-\sum_{i=k+1}^m \sigma_{Y,i}^2)/\sum_{i=1}^k w_i^2}.
\end{equation}
for any $\eta$ satisfying $0< \eta <\norm{Y}_F^2$.
\end{theorem}


Based on the above Theorems, {there exists a parameter $\lambda$ such that the minimizer of \eqref{constrained} is a SVT operator or WSVT operator with respect to $\lambda$.}
We can find a suitable regularization parameter $\lambda$ when the bound $\eta$ is given  such that a solution of \eqref{constrained} is one of the solution of \eqref{unconstrain}.
Since the regularization parameter $\lambda$ is the same as the thresholding parameter used in the singular value soft-thresholding operator, we can hence solve the unconstrained problem \eqref{unconstrain} easily.
By the equivalence of the two minimization problems, the constrained problem with the bound $\eta$ is also then solved.
We summarize the algorithm
for solving the nuclear norm minimization problem \eqref{constrained} in Algorithm \ref{LRMR-param}.

\begin{algorithm}
	\caption{{Low-rank} matrix recovery for constrained minimization problem}\label{LRMR-param}
	
	\begin{algorithmic}[1]\label{LRsvd}
		\REQUIRE{Observation matrix $Y$ and   $\eta$}
		\ENSURE{Estimated matrix $X$}
		\IF{$\norm{Y}_F^2\leq \eta $}
		\STATE{$X=\mathbf 0$}
		\ELSE
		\STATE{$Y=U\Sigma_YV^{\mathrm{T}}, \Sigma_Y=\diag{(\sigma_{Y,1},...,\sigma_{Y,m})}$}
		\STATE{Compute $b_j,j=1,...,m$ by \eqref{bj} or \eqref{wbj}}
		\STATE{Find $k$ such that $b_k>\eta\geq b_{k+1}$}
		\STATE{Compute $\lambda$ by \eqref{lambda2} or \eqref{wlambda}}
		\ENDIF
		\RETURN{$ X=\mathrm{SVT}_{\lambda}(Y)$ or $ X=\mathrm{WSVT}_{\lambda}(Y)$}
		
	\end{algorithmic}
	
\end{algorithm}



\subsection{Randomized Algorithm for Large-scale Problems}\label{RanSVD}
When the singular value thresholding operator is applied to recover the low-rank matrix, we need to calculate all the singular values of the observed  matrix $Y \in \R^{m\times n}$. The computation complexity is $O(m^2n)$ (here $m\leq n$). For large-scale matrices, singular values calculation is very time-consuming, which makes it infeasible in practical applications such as those in data sciences and image processing.
In order to avoid directly computing all the singular values, one can use the randomized SVD (RSVD) approach \cite{2010Finding,  Voronin2016A, oh2015fast}. The key  is to extract a small core matrix by finding an orthonormal matrix with the unitary invariant property.

{Let $\rank(X)=s$ and $s<\ell\leq m$. There exists an orthogonal matrix  $Q_\ell\in \R^{n\times \ell}$ with $Q_\ell^TQ_\ell=I_\ell$ and a matrix $A\in \R^{m\times \ell}$ such that $X$ has the factorization $X=AQ_\ell^T$.  Hence we have}
    \begin{equation}\label{randupperbound}
    \norm{Y-X}_F^2=\norm{YQ_\ell-A}_F^2 + a.
    \end{equation}
    Here $a=\norm{Y}_F^2-\norm{YQ_\ell}_F^2$.
    Obviously,  $YQ_\ell\in \R^{m\times \ell}$ is a matrix with the size less than $Y$.
    Notice that $f(A)=f(X)$ for $\rank(X)\leq \ell$, where $f(X)$ denotes the nuclear norm, the truncated nuclear norm or the weighted nuclear norm of $X$.
    Therefore, instead of solving the regularized minimization problem \eqref{unconstrain}, by changing the variable, we consider the following reduced minimization problem
    \begin{equation} \label{unconstrainedA}
    \min_{A} f(A)+\frac{1}{2\lambda}\norm{YQ_\ell-A}_F^2.
    \end{equation}
    The size of the matrix $YQ_\ell$ is smaller than that of $Y$.

{Here we consider the solution of the reduced problem}
    and once we obtain the optimal solution for $A$, the matrix $X$ can be recovered by  $X=AQ_\ell^T$.
    Therefore the computation complexity can be significantly improved when $\ell \ll n$.
    Similarly, for the constrained minimization problem, {the constraint can be reformulated as $\norm{YQ_\ell-A}_F^2\leq \eta-a$ according to \eqref{randupperbound}}.
    Thus we consider the following equivalent reduced problem
    \begin{equation}\label{constrainedA}
    \min_{\norm{YQ_\ell-A}_F^2\leq \eta-a}f(A).
    \end{equation}

    During the procedure to construct the reduced problem, we have assumed that  $Q_\ell$ is  known in advance and the rank of $X$ is less than the column number of $Q_\ell$.
    In some applications, however, the rank of $X$ is unknown.
    We can estimate the rank of $X$ according to the dominant components of the singular values of the noisy observation matrix $Y$.
    One way to obtain the dominant components of the singular values of  $Y$ is to construct an orthogonal matrix $Q_\ell$ such that the columns of $Q_\ell$ are the $\ell$ dominant right singular vectors of $Y$.
    The randomized method can be applied to efficiently compute $Q_\ell$.
    An ideal $Q_\ell$ should satisfy the condition that the number of columns of the orthogonal  matrix $Q_\ell$ is as small as possible such that the error $\norm{YQ_\ell Q_\ell^T- Y}$ is less than some desired tolerance.
    The details for the randomized method can be found in \cite{2010Finding}.

    Clearly, once $Q_\ell$ and $\eta$ are given, we can apply the same ideas developed in Subsection \ref{nns} and Subsection \ref{gwnns}  to obtain the regularization parameter $\lambda$ and hence solve the reduced problems \eqref{unconstrainedA} and \eqref{constrainedA}.

    {
    Now we analyze the computational complexity of using the standard SVD in Algorithm \ref{LRsvd}.  The SVD costs $O(m^2n)$.   Computing the sequences $b_j$  and $\lambda$ require $O(m^2)$ and $O(m)$, respectively. Then forming  the optimal $X$ needs $O(m^2n)$. Therefore, the total computational complexity of  Algorithm \ref{LRsvd} is $O(2m^2n+m^2+m)$.  If we use the randomized SVD, the total computational complexity is reduced to $O(2\ell^2 m+ \ell^2+\ell )$. Since $\ell \ll m$, the computation speed can be significantly improved. We note that before calculating the projection problem \eqref{constrainedA}, the orthogonal matrix $Q_\ell$ must be calculated.  The cost of obtaining $Q_\ell$ in $q$ iterations of the power method is $O(2qmn\ell+(m+n)\ell^2q)$, see \cite{2010Finding}. }

    \section{Applications}\label{numerical}
    The discrepancy principle is a classical method to choose the regularization parameter $\lambda$ for the regularized model \eqref{unconstrain} provided that an upper bound of the constraint $\|X - Y\|_F^2 \le \eta$ is given.
    In Section \ref{conreg}, we have developed the closed form formulas of $\lambda$ for the {NN  case and the GWNN case.} Once the parameter $\lambda$ is determined, the regularized solution of \eqref{unconstrain} can be reconstructed by the soft-thresholding formulas \eqref{svt} and \eqref{WSVT} directly. These solutions are also the solutions for the unconstrained problem \eqref{constrained} for the given $\eta$.
    The ideas can easily be extended to cover the {TNN} case (which is a shifted NN case, see  Subsection \ref{subsectionTNN}) and the randomized case, see Subsection \ref{RanSVD}.

    We illustrate the efficiency of solving the problems \eqref{constrained} and \eqref{unconstrain} using our approach in the following experiments.
    We denote the results obtained by the nuclear norm minimization problem \eqref{NNM} using our proposed regularization parameter $\lambda$ as ``NN-DP'' (see \eqref{lambda2}). When the truncated nuclear norm  or the generalized weighted nuclear norm is used, we
    denote our method as ``TNN-DP'' or ``GWNN-DP'' respectively.
    In the following tests, we set $r=1$ for TNN, and
    the weights  $w_i=(\sigma_{Y,i}+\epsilon)^{p-1}$ with $p=0.7$  {and $\epsilon=10^{-6}$} for GWNN. 

We compare the proposed algorithms with the following two approaches:
\begin{itemize}
    \item \textbf{SURE \cite{CandesUnbiased}:} The optimal soft thresholding parameter can be estimated by minimizing
    \begin{equation}\label{SURE}
    \operatorname{SURE}\left(\operatorname{SVT}_{\lambda}(Y)\right)=-m n \tau^{2}+\sum_{i=1}^m \min \left(\lambda^{2}, \sigma_{Y,i}^{2}\right)+2 \tau^{2} \operatorname{div}\left(\operatorname{SVT}_{\lambda}(Y)\right),
    \end{equation}
    where `div' is the divergence of $\operatorname{SVT}_{\lambda}$ {with respect to $Y$}. The solution is then given by
    $$
    \widehat{X}_{\operatorname{SURE}} =\sum_{i=1}^{m}(\sigma_{Y,i}-\lambda^\dagger)_+ \vu_i\vv_i^T,
    $$
   {where $\lambda^\dagger$ minimizes the SURE function. As suggested in \cite{CandesUnbiased}, $\lambda^\dagger$ is obtained by trial-and-error on 101 logarithmically equally spaced points between $10^{-1}$ and $10^7$.}
    The SURE approach can be easily generalized to the TNN and GWNN too. We denote the resulting methods by ``NN-SURE'', ``TNN-SURE'' and
    ``GWNN-SURE''.

    \item \textbf{HardT \cite{gavish2014the}:} The optimal  hard threshold is chosen as $\lambda=(4/\sqrt{3})\sqrt{n}\tau$ when the noise level $\tau$ is known or $0.2858\sigma_{Y,med}$ when $\tau$ is unknown. Here $\sigma_{Y,med}$ denotes the median value of the singular values of $Y$. The recovered matrix is given by
    $$
    \widehat{X}_{\operatorname{HardT}} =\sum_{i=1}^{m} \sigma_{Y,i}\mathbf{1}_{(\sigma_{Y,i}>\lambda)} \vu_i\vv_i^T.
    $$
\end{itemize}

\subsection{Simulations on Synthetic Data}
In the first set of experiments, we consider removing noise from synthetical matrices with size $m\times n$.
    The rank of the true low-rank matrix $X$ is determined by the  rank ratio $\rho$ and is generated by the product of two randomly sampling matrices  $M\in \R^{m \times s}$ and $N\in \R^{n \times s}$ from the standard uniform distribution, i.e., {$X_\dagger=MN^T$} where $s=\mathrm{round}(\rho \max(m,n))$. The observed matrix $Y$ is obtained by $Y={X_\dagger}+W$ where $W\in \R^{m\times n}$ is a matrix of zero-mean Gaussian white noise with  standard deviation $\tau$. We recover the matrix {$X_\dagger$} from its noisy observation matrix $Y$ by solving the constrained minimization problem \eqref{constrained}. The bound $\eta$ is determined by  \eqref{upperbound} with $c=1$. The Signal-to-Noise Ratio (SNR)\footnote{SNR is defined by
    $
    \text{SNR} = 10\log\left(\frac{\|X\|_F^2}{\|X-\widehat{X}\|_F^2}\right),
    $
    where $\widehat{X}$ is the estimated matrix.} is used to measure the quality of the recovered matrix.

    \subsubsection{ {Regularization Parameter} Selection}
    \begin{figure}[t]
    \centering

    {
    \begin{minipage}[t]{0.23\linewidth}
    \centering
    \includegraphics[height=2.8cm,width=3cm]{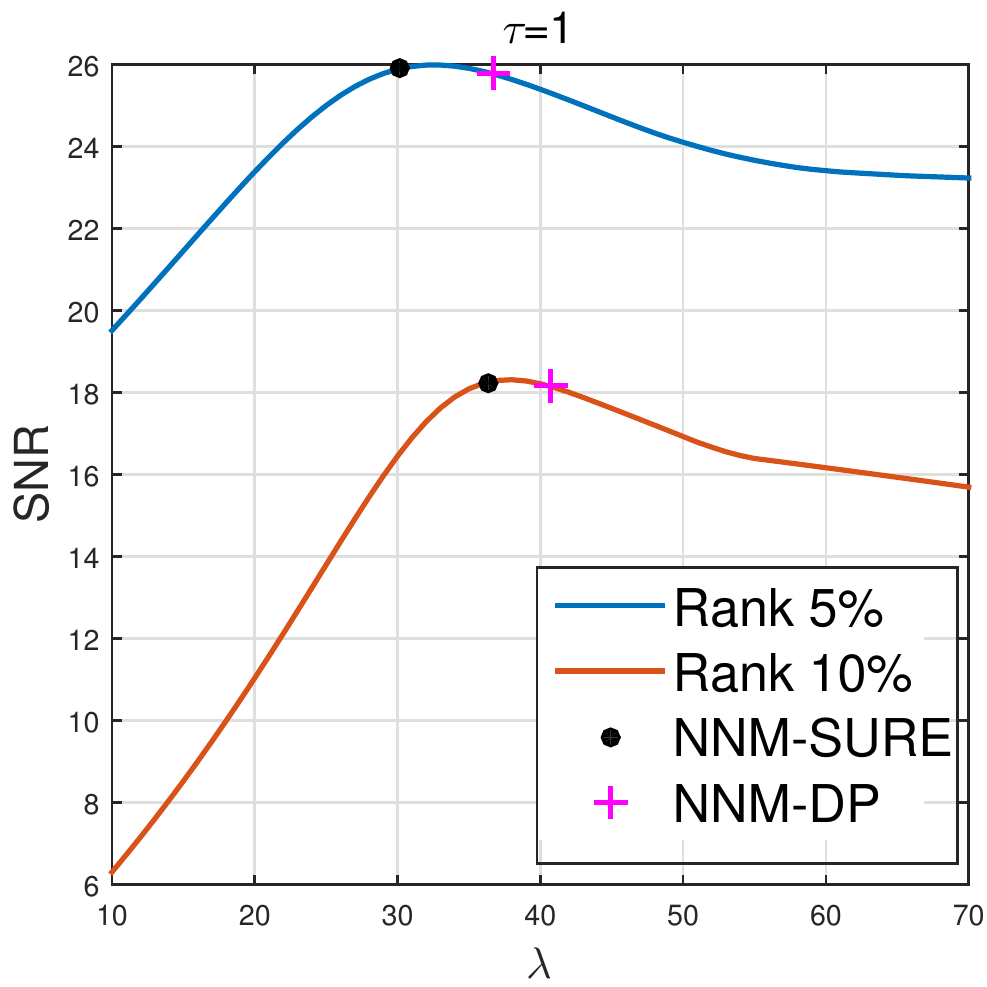}
    \end{minipage}
    }%
    {
    \begin{minipage}[t]{0.23\linewidth}
    \centering
    \includegraphics[height=2.8cm,width=3cm]{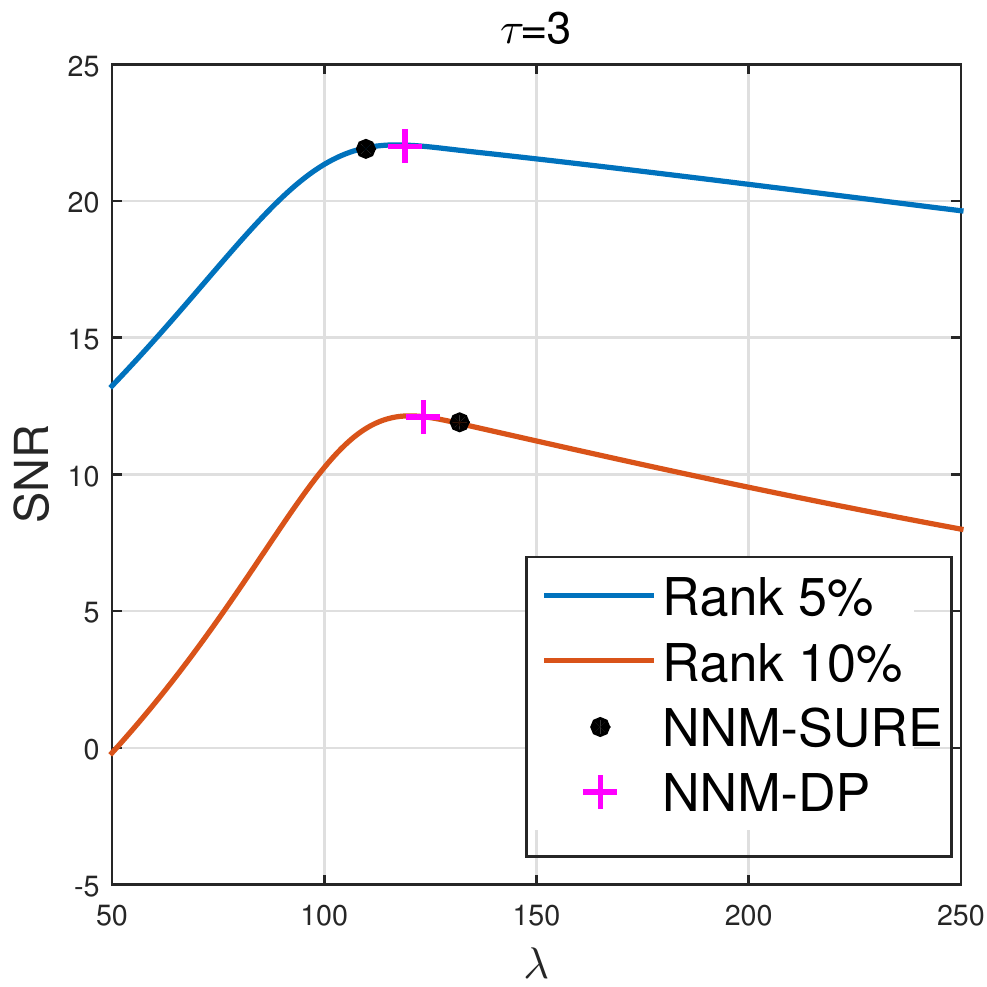}
    \end{minipage}
    }
    {
    \begin{minipage}[t]{0.23\linewidth}
    \centering
    \includegraphics[height=2.8cm,width=3cm]{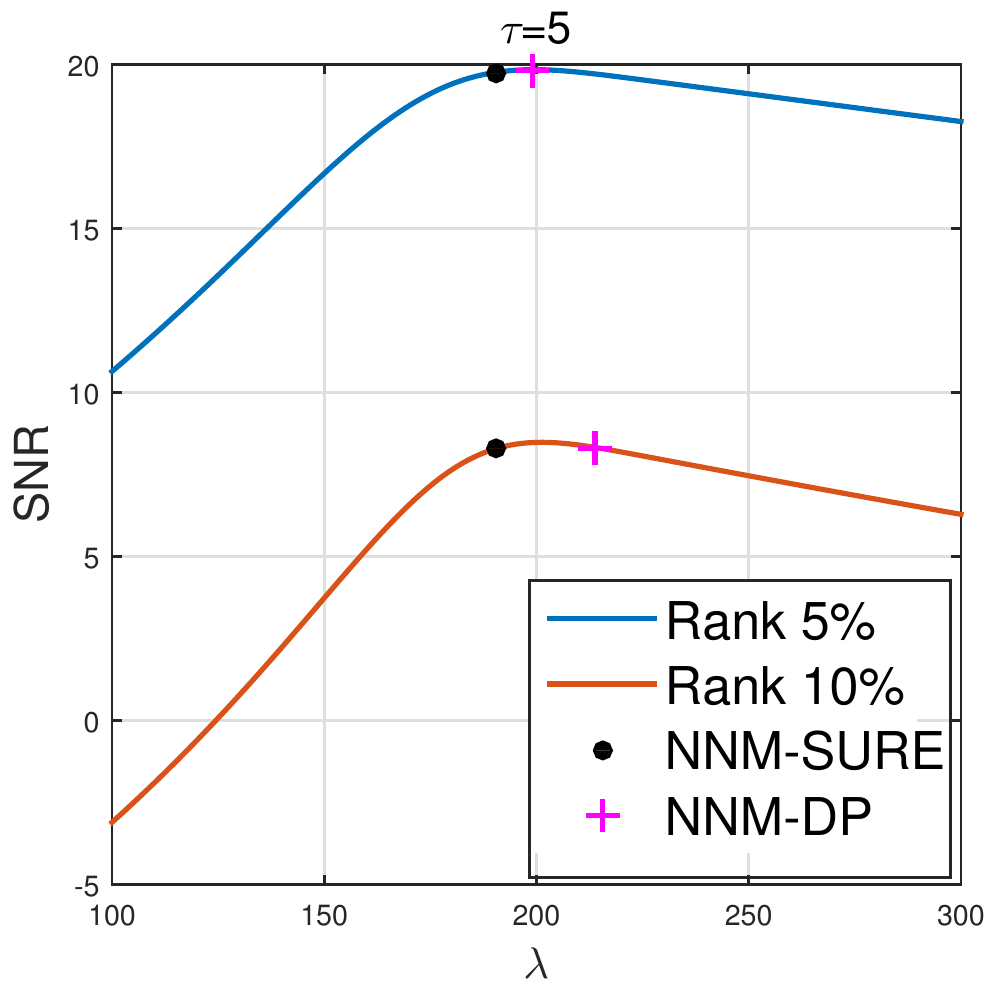}
    \end{minipage}
    }
    {
    \begin{minipage}[t]{0.23\linewidth}
    \centering
    \includegraphics[height=2.8cm,width=3cm]{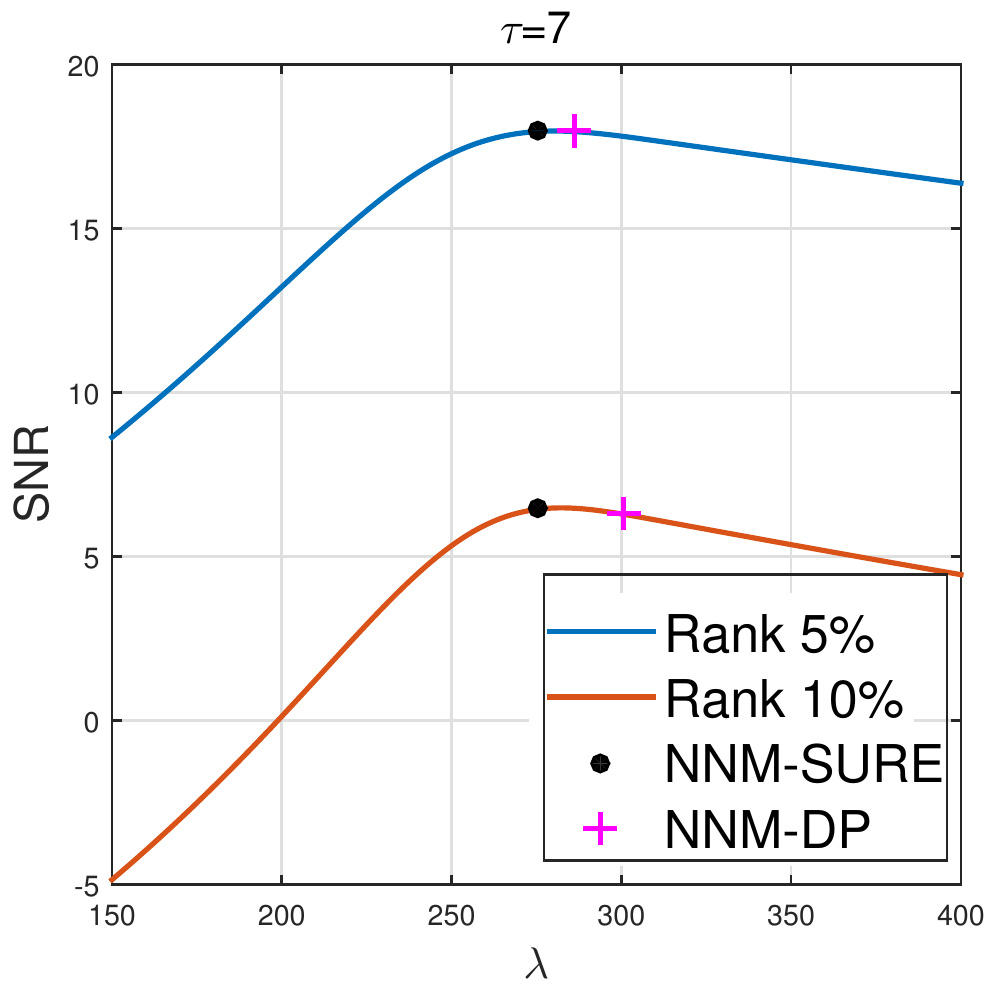}
    \end{minipage}
    }
    \begin{center}
    $m=n=500$
    \end{center}
    {
    \begin{minipage}[t]{0.23\linewidth}
    \centering
    \includegraphics[height=2.8cm,width=3cm]{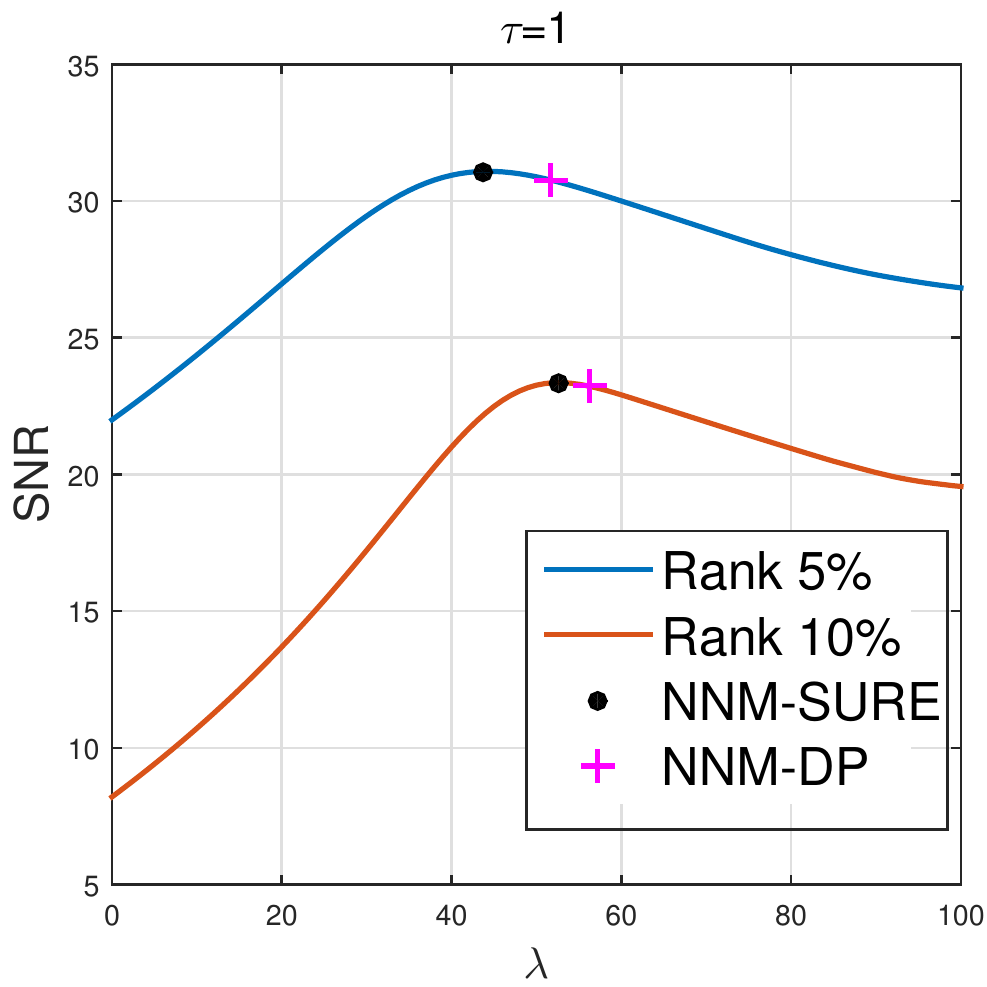}
    \end{minipage}
    }%
    {
    \begin{minipage}[t]{0.23\linewidth}
    \centering
    \includegraphics[height=2.8cm,width=3cm]{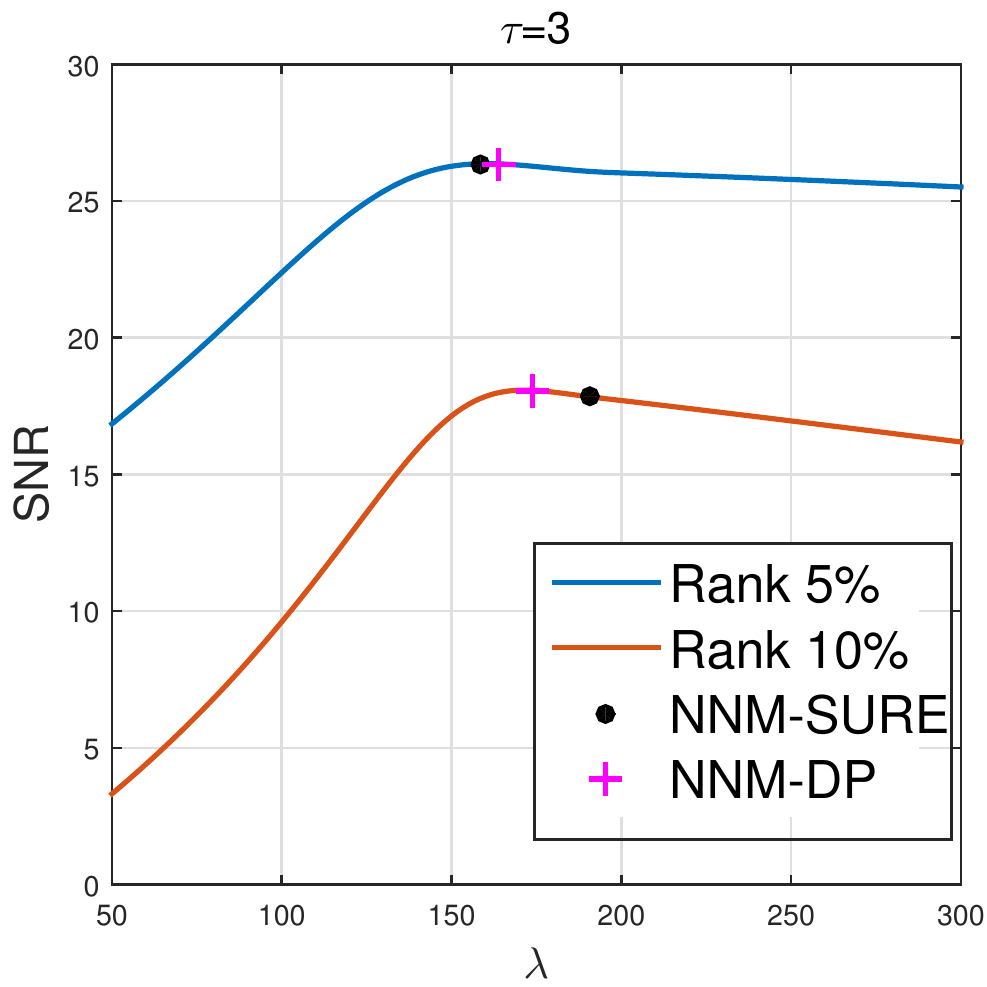}
    \end{minipage}
    }
    {
    \begin{minipage}[t]{0.23\linewidth}
    \centering
    \includegraphics[height=2.8cm,width=3cm]{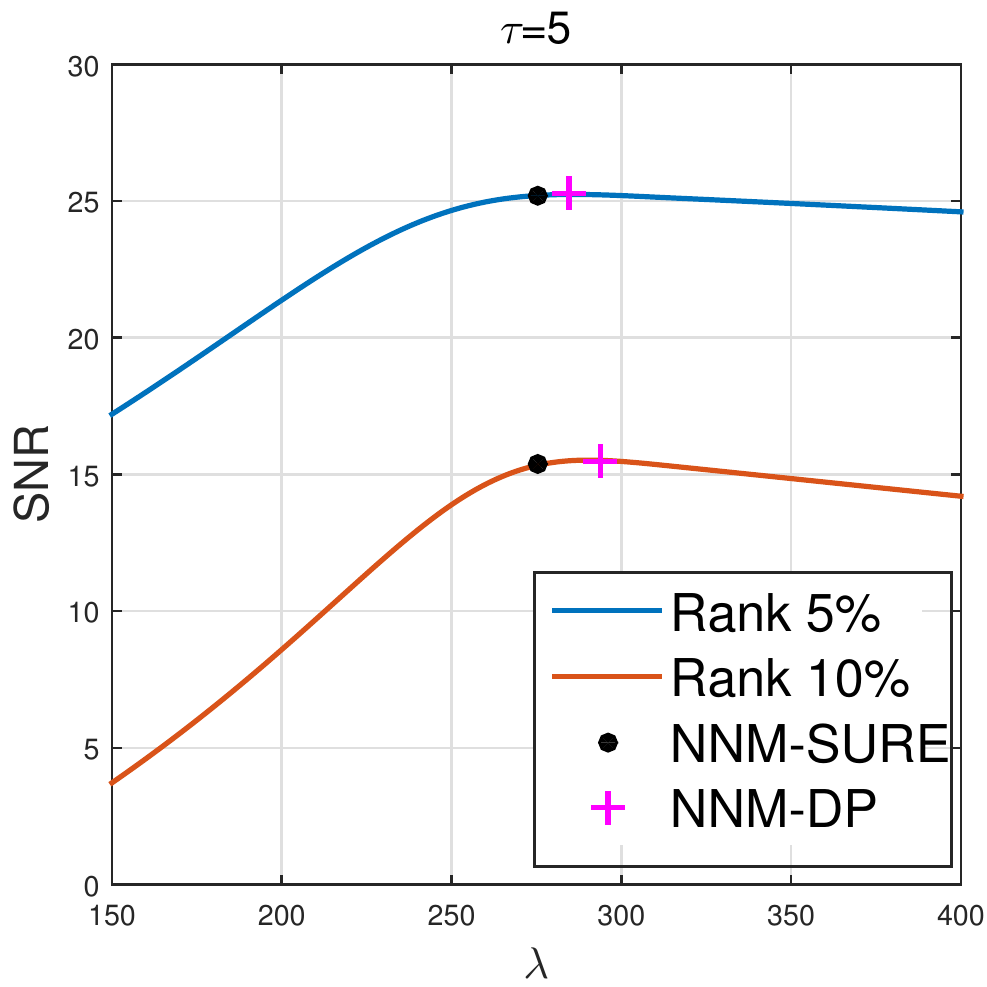}
    \end{minipage}
    }
    {
    \begin{minipage}[t]{0.23\linewidth}
    \centering
    \includegraphics[height=2.8cm,width=3cm]{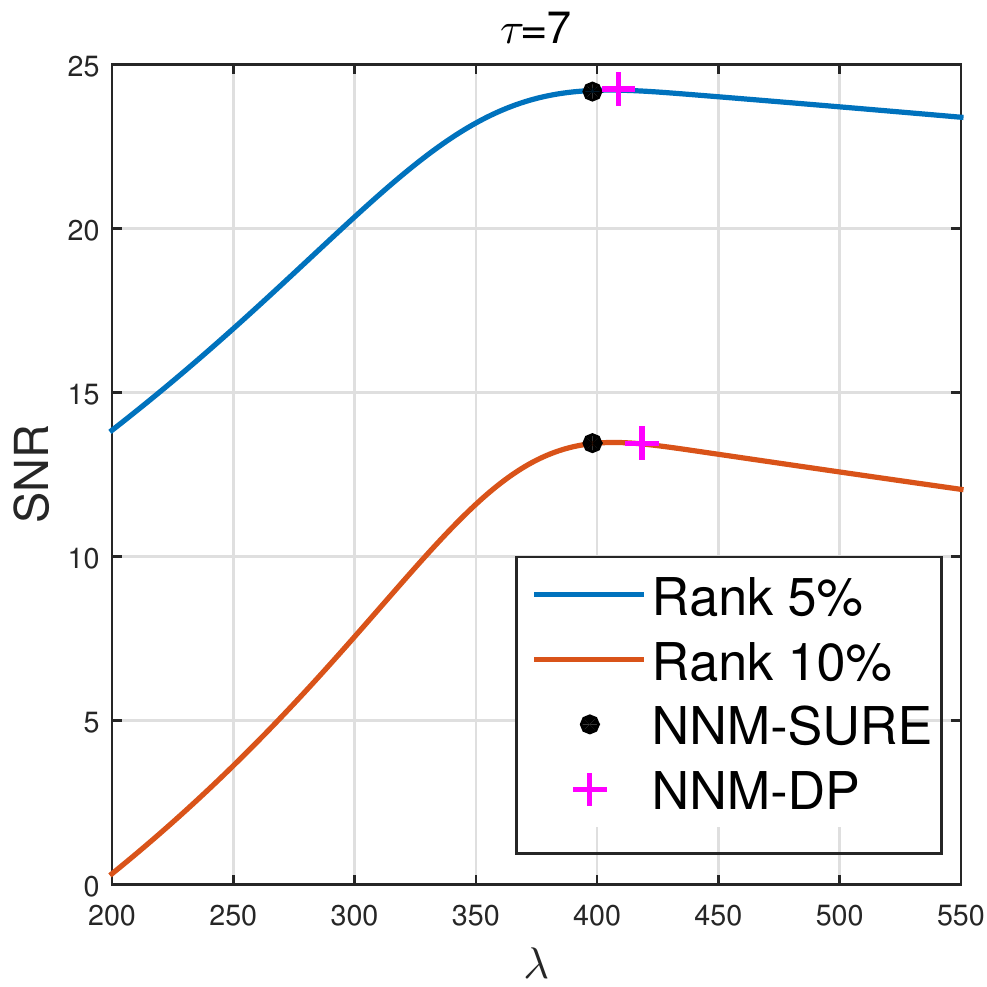}
    \end{minipage}
    }
    \begin{center}
    $m=n=1000$
    \end{center}
    {
    \begin{minipage}[t]{0.23\linewidth}
    \centering
        \includegraphics[height=2.8cm,width=3cm]{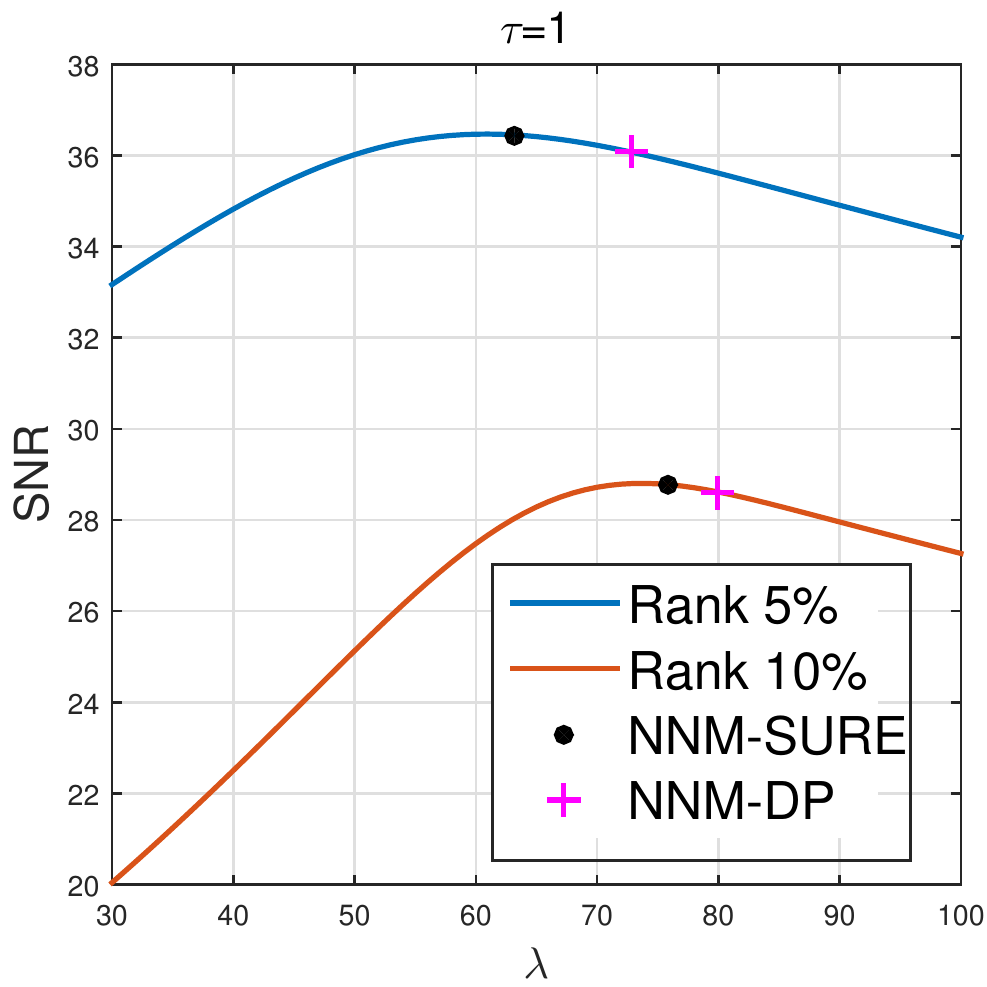}
    \end{minipage}
    }%
    {
    \begin{minipage}[t]{0.23\linewidth}
    \centering
    \includegraphics[height=2.8cm,width=3cm]{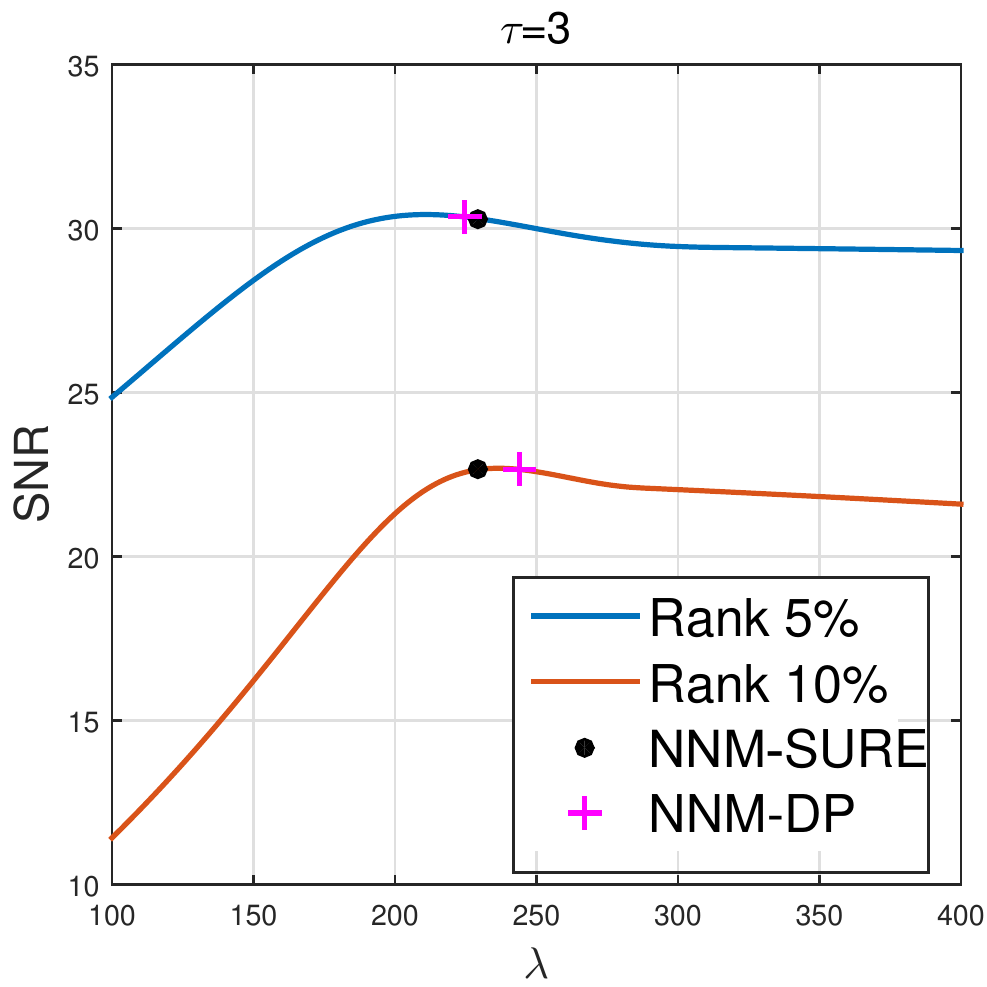}
    \end{minipage}
    }
    {
    \begin{minipage}[t]{0.23\linewidth}
    \centering
    \includegraphics[height=2.8cm,width=3cm]{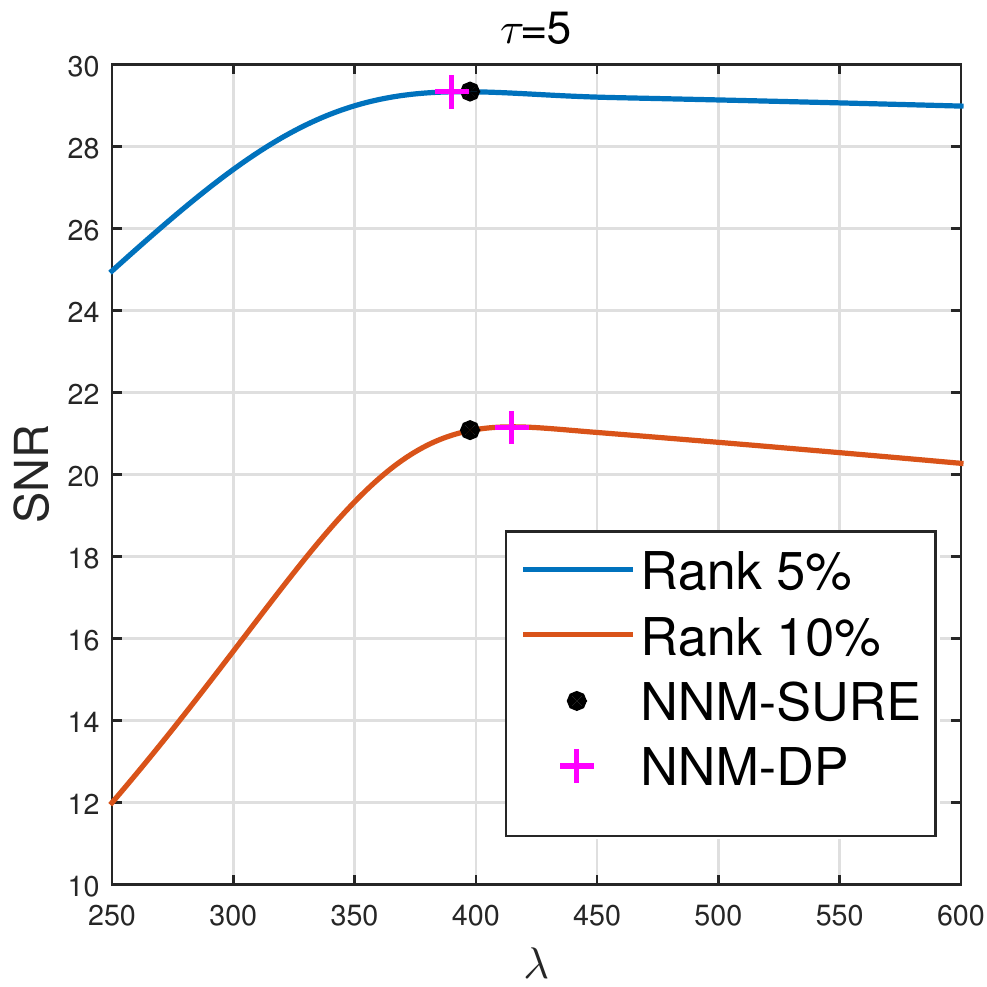}
    \end{minipage}
    }
    {
    \begin{minipage}[t]{0.23\linewidth}
    \centering
    \includegraphics[height=2.8cm,width=3cm]{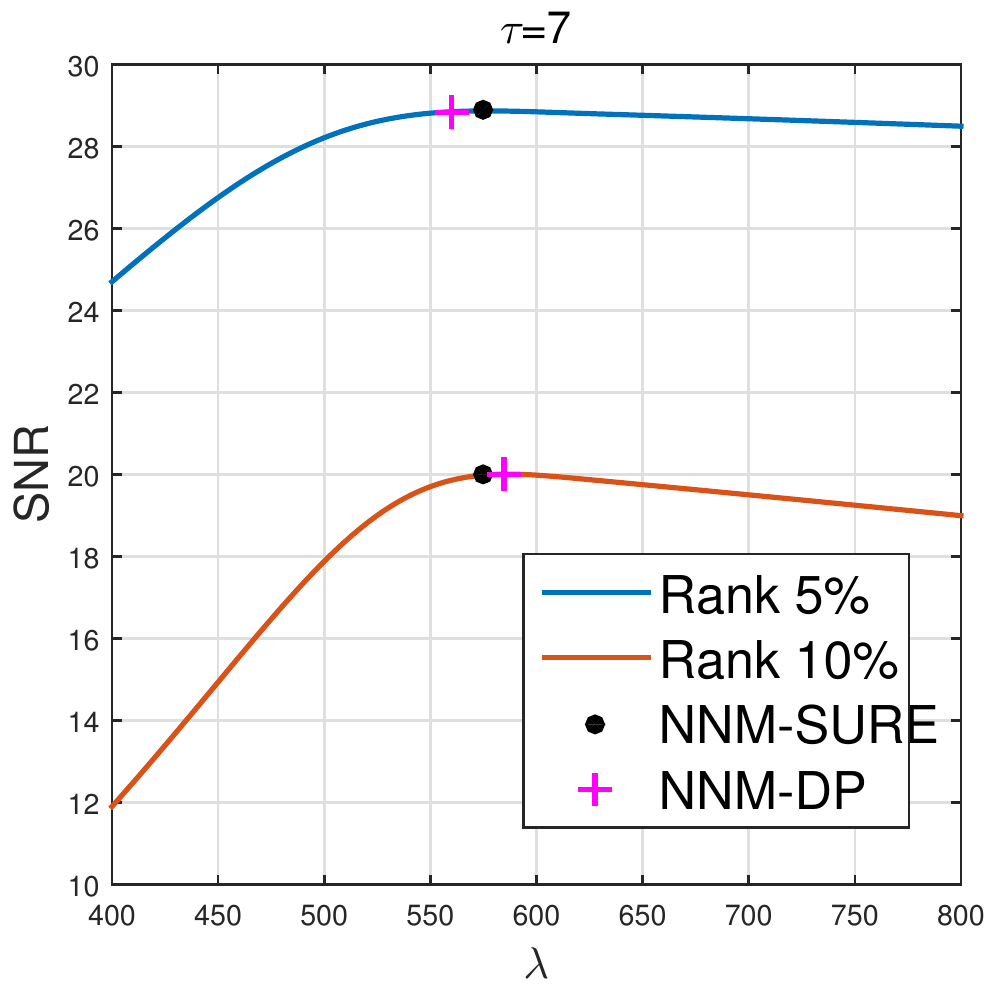}
    \end{minipage}
    }

    \begin{center}
    $m=n=2000$
    \end{center}
    \centering
    \caption{
    The  {threshold parameter} $\lambda$ versus SNR for the different rank rations, different noise levels, and different matrix size. The parameters determined by ``SURE'' and ``NNM-DP'' are represented by black ``$\bullet$'', and magenta ``$\textbf{+}$'', respectively.}
    \label{fig1}
    \end{figure}

    We first show that the selected regularization parameter by our method is a good one.
    We consider the standard nuclear norm as the penalty function.
    Rank ratio  is set to $\rho=5\%$ and $10\%$ and {the noise standard derivation level} is set to $\tau=1,3,5,7$ respectively.
    The size of the matrix is set to $m=n=500, 1000,2000$ respectively.

    The plots in  Figure \ref{fig1} show the curves of the SNR against the  {regularization parameter} $\lambda$ for different rank ratios and different noise levels.
    The SNR obtained by our approach {(i.e., solving the minimization problem \eqref{unconstrain} using Algorithm \ref{LRsvd})} and by the ``SURE'' are marked by black ``$\bullet$'' and magenta ``$+$'' respectively.
    We observe in the figures that even if the noise level is the same, the optimal  {regularization parameter} varies with the size of the matrix.
    We also observe that the SNRs obtained by ``SURE'' are close to the maximum of the curve in most cases because the parameter obtained by ``SURE'' has the minimum MSE within a given range of the parameter.

    The {regularization parameter} obtained by our method tend to be slightly larger than the optimal one. This is consistent with the observations that the selection of discrepancy principle with $c=1$ usually yields an over smoothed result \cite{1989Image, 1992Methods}.
    In order to obtain the optimal SNRs, {one can} choose a small value for $c$ (i.e., $c<1$) or decrease the  {regularization parameter}.

    We remark that when the size of the matrices is big and the rank ratio is small (see the last row in the figure), our method gives better {$\lambda$}.
    {We also note that ``SURE'' obtains the regularization parameters by solving the minimization problem (\ref{SURE}) by trial-and-error,} which is time-consuming,
    while our method provides a closed form formula for $\lambda$ directly (see \eqref{lambda2} and \eqref{wlambda}). We will see below that the CPU running time of our method is indeed much faster than that of ``SURE''.

    \subsubsection{Comparisons of Denoising Results}
    In this subsection, we compare the different methods in denoising. Here two tasks are performed.
    The first one is based on the {full} observation matrix, and the second one is based on the randomized projected matrix after dimensionality reduction. In the randomized case, the size $\ell$ of the projection matrix $Q_\ell$ (see Subsection \ref{RanSVD}) is specified in advance for the best rank-$\ell$ approximation. Here we take $\ell={s}+5$, where $s=\mathrm{round}(\rho \max(m,n))$.
    According to \eqref{randupperbound}, the bound $\hat{\eta}$ for the matrix $\widehat{Y}_\ell=YQ_\ell$ is given by
    $$\hat{\eta}=\eta-\norm{Y}_F^2+\norm{\widehat{Y}_\ell}_F^2.$$
    Note that ``HardT'' and ``SURE'' need the noise level of the low-dimensional matrix $\widehat{Y}_\ell$ as an input argument.
    Its noise level can be estimated by   $\hat{\tau}= (c mn\tau^2 -\norm{Y}_F^2+\norm{\widehat{Y}_\ell}_F^2)/(n\ell)$ if the noise level $\tau$ of $Y$ is given. Otherwise, one needs to estimate it either implicitly or explicitly  from $\widehat{Y}_\ell$.

    We set the matrix size  $m=n=500, 1000, 2000,4000$, and {the standard deviation $\tau=3,5$ respectively.}
    The rank ratio $\rho$ is set to $ 1\%, 5\%, 10\%, 20\%, 40\%$, and the rank of the clean matrix {$X_\dagger$} is fixed to $\rho \max(m,n)$.
    We consider four different types of objective functions: rank$(X)$, $\|X\|_*$, $\|X\|_r$ {(with $r=1$)} and $\|X\|_{\vw,*}$. In fact, the penalty functions $\|X\|_*$, $\|X\|_r$ and $\|X\|_{\vw,*}$ are relaxations of the function rank$(X)$.
    Here the SNRs and CPU running times are computed by running the experiments 10 times and taking the average of the 10 tests.
    {We use boldface to mark the best method for a given {matrix} amongst all the different regularization norms and methods; and we use the italic font to mark the best method for the given regularization norm.}

    \begin{table}[t]
    \centering
    \caption{Comparison of SNR(dB) for different methods with standard algorithm.} \large
    \scalebox{0.6}{
    \begin{tabular}{M{1.5cm}M{1cm}M{1cm}|M{1.8cm}|M{1.8cm}M{1.8cm}|M{1.8cm}M{1.8cm}|M{1.8cm}M{1.8cm}}
    \hline
    &       & $f(X)$  & Rank$(X)$  & \multicolumn{2}{c|}{$\|X\|_*$} & \multicolumn{2}{c|}{$\|X\|_r$} &\multicolumn{2}{c}{ $\|X\|_{\vw,*}$} \\ \hline
    $m=n$&   $\tau$    & $\rho$  &  HardT &  NN-SURE &  NN-DP &  TNN-SURE & TNN-DP &  GWNN-SURE & GWNN-DP \\ \hline
    500
    &3
    & 1\%   & 14.47  & 11.78  & \textit{11.96}  & \textit{\textbf{14.47 }} & \textit{\textbf{14.47 }} & \textit{13.49}& 13.05  \\
    &       & 5\%   & \textbf{22.99}  & 21.84  & \textit{21.94}  & \textit{\textbf{22.99 }} & \textit{\textbf{22.99 }}&\textit{22.82} & \textit{22.82}  \\
    &       & 10\%  & 26.25  & 25.69  & \textit{25.72}  & 26.25  & \textit{\textbf{26.29 }} & \textit{26.22}& 26.21  \\
    &       & 20\%  & 29.43  &\textit{29.30}  & \textsl{29.30}  & \textit{\textbf{29.59 }} & \textit{\textbf{29.59 }}&29.51 & \textit{29.54}  \\
    &       & 40\%  & 32.48  & 32.72  & \textit{32.75}  & 32.84  & \textit{\textbf{32.89 }} & \textit{32.84} &\textit{32.84}  \\
    &5 & 1\%   & \textbf{11.35 } &\textit{8.53}  & 8.51  & \textit{\textbf{11.35 }} & \textit{\textbf{11.35 }} &\textit{9.96}& 9.31  \\
    &       & 5\%   & \textbf{21.80 } & 19.71  & \textit{19.75}  & \textit{\textbf{21.80 }} & 21.79  & \textit{21.30}& 21.26  \\
    &       & 10\%  & \textbf{25.57 } & 24.16  & \textit{24.20}  & \textit{\textbf{25.57 }} & 25.53  & \textit{25.34}& 25.32  \\
    &       & 20\%  & \textbf{29.06 } & 28.26  & \textit{28.27}  & \textit{\textbf{29.06 }} & 29.03  & \textit{28.97}&28.94  \\
    &       & 40\%  & 32.29  & \textit{31.92}  & 31.91  & \textit{\textbf{32.31} } & \textit{\textbf{32.31 }} & \textit{32.27}&32.26  \\ \hline
    \multicolumn{3}{c|}{Average SNR} & 24.57  & 23.39  & \textit{23.43}  & \textit{\textbf{24.62 }} & \textit{\textbf{24.62 }}&\textit{24.27} & 24.16  \\\hline
    \multicolumn{3}{c|}{ {Average time ({sec})}} &  {0.054}  &  {1.124}  &  {0.059}  &  {1.178}  &  {0.060}  & 1.144& {0.058}  \\
    \hline
    1000
    & 3
    & 1\%   & 19.22  & 17.69  & \textit{17.94}  & 19.22  & \textit{\textbf{19.25 }} &18.90 & \textit{18.97}  \\
    &       & 5\%   & 26.46  & \textit{26.34}  & \textit{26.34}  & 26.64  & \textit{\textbf{26.65 }} & 26.61 & \textit{26.62}  \\
    &       & 10\%  & 29.51  & 29.71  & \textit{29.72}  & 29.87  & \textit{\textbf{29.88 }} &  29.84 &\textit{29.85}   \\
    &       & 20\%  & 32.54  & \textit{33.14}  & 33.10  & \textit{\textbf{33.20 }} & 33.18  &\textit{33.16} & 33.15  \\
    &       & 40\%  & 35.39  & \textit{36.67}  & 36.59  & \textit{\textbf{36.71 }} & 36.63  &  \textit{36.65}& 36.59  \\
    & 5 & 1\%   & \textbf{17.76 } & 15.34  & \textit{15.48}  & \textit{\textbf{17.76 }} & \textit{\textbf{17.76 }} &   \textit{16.98} &16.90  \\
    &       & 5\%   & \textbf{26.08 } & 25.22  & \textit{25.25}  & \textit{\textbf{26.08 }} & 26.05 & \textit{25.97} & 25.95  \\
    &       & 10\%  & \textbf{29.32 } & \textit{28.89}  & 28.88  & \textit{\textbf{29.32 }} & 29.31 &\textit{29.30} & 29.26  \\
    &       & 20\%  & 32.44  & \textit{32.29}  & \textit{32.29}  & \textit{\textbf{32.52 }} & 32.51 &\textit{32.47} & 32.46  \\
    &       & 40\%  & 35.51  & 35.57  & \textit{35.63}  & 35.68  & \textit{\textbf{35.73 }} & 35.67 &\textit{35.69}  \\ \hline
    \multicolumn{3}{c|}{Average SNR} & 28.42  & 28.09  & \textit{28.12}  & \textit{\textbf{28.70 }} & \textit{\textbf{28.70 }}& \textit{28.56} & 28.54  \\ \hline
    \multicolumn{3}{c|}{ {Average time  (sec)}} &  {0.240}  &  {5.646}  &  {0.261}  &  {5.992}  &  {0.267}  & 5.895 & {0.249}  \\
    \hline
    2000 &
    3& 1\%  & 22.64  & 22.62  & \textit{22.63}  & 23.02  & \textit{\textbf{23.09} } & 23.04 &\textit{23.08}  \\
    &       & 5\%   & 29.54  & 30.27  & \textit{30.31}  & 30.37  & \textit{30.40}  & \textbf{\textit{30.50}} &{30.44 } \\
    &       & 10\%  & 32.48  & \textit{33.79}  & 33.67  & \textit{33.83}  & 33.71  & \textbf{\textit{33.80}} &{33.74 } \\
    &       & 20\%  & 35.64  & \textit{37.27}  & 37.11  & \textit{{37.28 }} & 37.13 &\textbf{\textit{37.33}} & 37.14  \\
    &       & 40\%  & 39.12  & \textit{41.06}  & 40.74  & \textit{\textbf{41.07 }} & 40.74 &\textit{41.05} & 40.74  \\
    & 5 & 1\%   & \textbf{22.17 } & 21.07  & \textit{21.15}  & \textit{\textbf{22.17 }} & 22.16 &21.99 & \textit{22.00}  \\
    &       & 5\%   & 29.46  & \textit{29.30}  & 29.29  &\textit{\textbf{29.53 }} &\textit{\textbf{29.53}} &29.46 & \textit{29.49}  \\
    &       & 10\%  & 32.52  & 32.54  & \textit{32.59}  & 32.67  & \textit{\textbf{32.70 }} & \textit{32.67} &\textit{32.67}  \\
    &       & 20\%  & 35.55  & 35.86  & \textit{35.87}  & 35.91  & \textit{\textbf{35.92 }} & 35.85 & \textit{35.89}  \\
    &       & 40\%  & 38.55  & \textit{39.26}  & 39.25  & \textit{39.29}  & \textit{\textbf{39.27 }} & \textit{39.26} & 39.23  \\\hline
    \multicolumn{3}{c|}{Average SNR} & 31.77  & \textit{32.30}  & 32.26  & \textit{\textbf{32.51 }} & 32.47  & \textit{32.50} & 32.44  \\ \hline
    \multicolumn{3}{c|}{ {Average time (sec)}} &  {1.572}  &  {28.956}  &  {1.571}  &  {31.651}  &  {1.583}  & 30.140 & {1.562}  \\
    \hline
    4000 & 3
    & 1\%   & 26.06  & \textit{27.19}  & 27.12  & \textit{{27.35 }} & 27.27  & \textbf{\textit{27.58}} &27.48  \\
    &       & 5\%   & 33.37  & \textit{34.85}  & 34.66  & \textit{{34.88 }} & 34.68 & \textbf{\textit{35.05}} & 34.85  \\
    &       & 10\%  & 36.70  & \textit{38.23}  & 38.03  & \textit{{38.25 }} & 38.04 & \textbf{\textit{38.42}} & 38.16  \\
    &       & 20\%  & 40.21  & \textit{41.95}  & 41.54  & \textit{{41.96 }} & 41.55  & \textbf{\textit{42.05}}& 41.62  \\
    &       & 40\%  & 43.99  & \textit{\textbf{45.90 }} & 45.28  & \textit{{45.90 }} & 45.29 &\textbf{\textit{45.92}} & 45.32  \\
    & 5 & 1\%   & 25.57  & \textit{25.44}  & \textit{25.44}  & \textit{\textbf{25.75 }} & \textit{\textbf{25.75 }}&25.72 & \textit{25.73}  \\
    &       & 5\%   & 32.55  & 32.87  & \textit{32.97}  & 32.92  & \textit{{33.03 }} & 33.00 & \textbf{\textit{33.04}}  \\
    &       & 10\%  & 35.56  & \textit{36.31}  & 36.25  & \textit{\textbf{36.33 }} & 36.28  & \textit{36.32} & 36.28  \\
    &       & 20\%  & 38.44  & \textit{39.69}  & 39.63  & \textit{\textbf{39.70 }} & 39.64  & \textit{39.70} & 39.64  \\
    &       & 40\%  & 41.61  & \textit{43.38}  & 43.17  & \textit{\textbf{43.39 }} & 43.18  & \textit{43.32} & 43.16  \\\hline
    \multicolumn{3}{c|}{Average SNR} & 35.41  & \textit{36.58}  & 36.41  & \textit{{36.64 }} & 36.47  & \textbf{\textit{36.71}} & 36.53  \\\hline
    \multicolumn{3}{c|}{ {Average time (sec)}} &  {{14.232}}  &  {159.368}  &  \textit{14.313}  &  {161.864}  &  \textit{14.355}  & 160.932& {14.334}  \\
    \hline
    \end{tabular}}%
    \label{Table1}%
    \end{table}%

    \begin{table}[t]
    \centering
    \caption{Comparison of SNR(dB) for different methods with randomized algorithm.} \large
    \scalebox{0.6}{
    \begin{tabular}{M{1.5cm}M{1cm}M{1cm}|M{1.8cm}|M{1.8cm}M{1.8cm}|M{1.8cm}M{1.8cm}|M{1.8cm}M{1.8cm}}
    \hline
    &       & $f(X)$  & Rank$(X)$  & \multicolumn{2}{c|}{$\|X\|_*$} & \multicolumn{2}{c|}{$\|X\|_r$} &\multicolumn{2}{c}{ $\|X\|_{\vw,*}$} \\ \hline
    $m=n$&   $\tau$    & $\rho$  &  HardT &  NN-SURE &  NN-DP &  TNN-SURE & TNN-DP &  GWNN-SURE & GWNN-DP \\ \hline

    500 & 3
    & 1\%   & \textbf{14.47 } & 11.78 & \textit{12.28}  & \textit{\textbf{14.47 }} & \textit{\textbf{14.47 }} & \textit{13.62} &13.15  \\
    &       & 5\%   & \textbf{22.99 } & 21.75  & \textit{22.02}  & \textit{\textbf{22.99 }} & 22.97  & 22.75 & \textit{22.81}  \\
    &       & 10\%  & \textbf{26.25 } & 25.52  & \textit{25.72}  & \textit{\textbf{26.25 }} & \textit{\textbf{26.25 }} &26.15 & \textit{26.18}  \\
    &       & 20\%  & 29.43  & 29.05  & \textit{29.29}  & {29.43}  & \textit{\textbf{29.57 }} & 29.39 & \textit{29.52 } \\
    &       & 40\%  & 32.48  & 32.19  & \textit{32.74}  & 32.48  & \textit{\textbf{32.88 }} & 32.47& \textit{32.84 } \\
    &5 & 1\%   & \textbf{11.35 } & 8.19  & \textit{8.79}  & \textit{\textbf{11.35 }} & \textit{\textbf{11.35 }} &\textit{9.97} & 9.45  \\
    &       & 5\%   & \textbf{21.80 } & 19.39  & \textit{19.90}  & \textit{\textbf{21.80 }} & 21.79  & \textit{21.30} &21.29  \\
    &       & 10\%  & \textbf{25.57 } & 23.90  & \textit{24.25}  & \textit{\textbf{25.57 }} & 25.52  & \textit{25.34} &25.32  \\
    &       & 20\%  & \textbf{29.06 } & 28.04  & \textit{28.27}  & \textit{\textbf{29.06 }} & 29.03  & 28.93 & \textit{28.94}  \\
    &       & 40\%  & 32.29  & 31.72  & \textit{31.91}  &  32.29  & \textit{\textbf{32.31 }} & 32.25 & \textit{32.26 } \\ \hline
    \multicolumn{3}{c|}{Average SNR} & 24.57  & 23.15  & \textit{23.52}  & 24.57  & \textit{\textbf{24.61 }} & \textit{24.22}&24.18  \\\hline
    \multicolumn{3}{c|}{ {Average time (sec)}} &  {0.005}  &  {0.116}  &  {0.006}  &  {0.123}  &  {0.007}  & 0.117& {0.006}  \\
    \hline
    1000 & 3
    & 1\%   & \textbf{19.22 } & 17.71  & \textit{18.11}  & \textit{\textbf{19.22}}  & 19.18  & 18.90& \textit{18.95}  \\
    &       & 5\%   & 26.46  & 26.00  & \textit{26.25}  & 26.46  & \textit{\textbf{26.51 }} & 26.43& \textit{26.47}  \\
    &       & 10\%  & 29.51  & 29.32  & \textit{29.64}  & 29.51  & \textit{\textbf{29.78 }} & 29.50 & \textit{29.75}  \\
    &       & 20\%  & 32.54  & 32.53  & \textit{33.07}  & 32.54  & \textit{\textbf{33.14 }} & 32.54 & \textit{33.12}  \\
    &       & 40\%  & 35.56  & 35.76  & \textit{36.59}  & 35.82  & \textit{\textbf{36.62 }} &  35.89 &  \textit{36.59}  \\
    & 5 & 1\%   & \textbf{17.76 } & 15.56  & \textit{15.80}  & \textit{\textbf{17.76}}  & 17.75  &\textit{17.19} & 16.99  \\
    &       & 5\%   & \textbf{26.08 } & 25.04  & \textit{25.30}  & \textbf{\textit{26.08}}  & 26.01 & \textit{25.97} & 25.92  \\
    &       & 10\%  & \textbf{29.32 } & 28.72  & \textit{28.88}  & \textit{\textbf{29.32}}  & 29.28 &\textit{29.27} & 29.23  \\
    &       & 20\%  & 32.44  & 32.12  & \textit{32.28}  & 32.44  & \textit{\textbf{32.49} } & 32.41 &\textit{32.45}  \\
    &       & 40\%  & 35.51  & 35.35  & \textit{35.63}  & 35.51  & \textit{\textbf{35.73 }} & 35.51 &\textit{35.68}  \\\hline
    \multicolumn{3}{c|}{Average SNR} & 28.44  & 27.81  & \textit{28.16}  & 28.47  & \textit{\textbf{28.65 }} & 28.36& \textit{28.52} \\\hline
    \multicolumn{3}{c|}{ {Average time (sec)}} &  {0.020}  &  {0.366}  &  {0.022}  &  {0.394}  &  {0.023}  &  0.393 &{0.021}  \\
    \hline
    2000 & 3
    & 1\%   & 22.64 & 22.19 & \textit{22.25} & {22.49} & \textit{\textbf{22.55}} & 22.32 & \textit{22.52} \\
    &       & 5\%   & 29.50  & 29.69 &  \textit{30.01} & 29.73 & \textit{30.08} &29.49 &\textit{\textbf{30.09}} \\
    &       & 10\%  & 32.54 & 33.33 & \textit{33.51} & {33.36} & \textit{33.55} & 33.15 &\textit{\textbf{33.56}} \\
    &       & 20\%  & 35.64 & 36.83 & \textit{37.06} & 36.84 & \textit{37.08} & 37.02 &\textit{\textbf{37.09}} \\
    &       & 40\%  & 38.95 & 40.72 & \textit{40.73} & 40.73 & \textit{\textbf{40.74}} & 40.54 & \textit{40.73} \\
    & 5 & 1\%   & \textbf{22.17} & 21.25 & \textit{21.27} & \textit{\textbf{22.17}} & 22.09 & \textit{21.99}& 21.96 \\
    &       & 5\%   & \textbf{29.46} & 29.12 & \textit{29.23} & \textit{\textbf{29.46}} & 29.43 &\textit{29.45}& 29.39 \\
    &       & 10\%  & 32.52 & 32.35 & \textit{32.52} & 32.52 & \textit{\textbf{32.63}} & 32.51 & \textit{32.59} \\
    &       & 20\%  & 35.55 & 35.46 & \textit{35.84} & 35.55 & \textit{\textbf{35.89}} & 35.54 & \textit{35.86} \\
    &       & 40\%  & 38.57 & 38.56 & \textit{39.24} & 38.57 & \textit{\textbf{39.27}} & 38.57 & \textit{39.23} \\\hline
    \multicolumn{3}{c|}{Average SNR} & 31.75  & 31.95  & \textit{32.17}  & 32.14  & \textit{\textbf{32.33}}&32.06 & \textit{32.30}  \\\hline
    \multicolumn{3}{c|}{ {Average time (sec)}} &  {0.089}  &  {1.810}  &  {0.095}  &  {2.040}  &  {0.094}  & 2.016& {0.092}  \\ \hline
    4000 & 3
    &       1\%      & 23.82 & 25.39  & \textit{26.14}  & 25.49  & \textit{26.21}  & 25.51 &\textit{\textbf{26.22 }} \\
    &       & 5\%   & 32.88  & 33.98  & \textit{34.34}  & 33.99  & \textit{34.36}  & 33.94 &\textit{\textbf{34.40 }} \\
    &       & 10\%  & 36.60  & 37.71  & \textit{37.89}  & 37.72  & \textit{37.90}  & 37.68&\textit{\textbf{37.96 }} \\
    &       & 20\%  & 40.11  & \textit{41.62}  & 41.51  & \textit{\textbf{41.62 }} & 41.51 & \textit{41.60}& 41.56  \\
    &       & 40\%  & 43.70  & \textit{45.66}  & 45.28  & \textit{\textbf{45.66 }} & 45.28  &\textit{ 45.64}&45.31  \\
    & 5 & 1\%   & 25.57  & 24.69  & \textit{25.16}  & 25.02  & \textit{\textbf{25.39 }} & \textit{24.92}& 25.37  \\
    &       & 5\%   & 32.55  & 32.33  & \textit{32.73}  & 32.36  & \textit{\textbf{32.78 }} & 32.48&\textit{32.77}  \\
    &       & 10\%  & 35.56  & 35.87  & \textit{36.11}  & 35.88  & \textit{\textbf{36.14 }} & 35.43 &\textit{36.13}  \\
    &       & 20\%  & 38.58  & 39.18  & \textit{39.59}  & 39.19  & \textit{\textbf{39.60 }} & 39.29 & \textit{39.59 } \\
    &       & 40\%  & 41.62  & 42.95  & \textit{43.17}  & 42.95  & \textit{\textbf{43.18 }} & 42.98 & \textit{43.15}  \\\hline
    \multicolumn{3}{c|}{Average SNR} & 35.10  & 35.94  & \textit{36.19}  & 35.99  & \textit{36.24}  &35.95& {\textit{\textbf{36.25 }}} \\\hline
    \multicolumn{3}{c|}{ {Average time (sec)}} &  \textbf{0.500}  &  {9.363}  &  {{\textit{0.509}}}  &  {9.532}  &  {{\textit{0.522}}} & 9.578 &  {0.516}  \\
    \hline
    \end{tabular}}%
    \label{Table2}%
    \end{table}%

    The SNRs obtained by different methods are shown in Table \ref{Table1} for the standard non-randomized algorithms.
    From the table, we see that the SNRs obtained by the rank function are larger than those obtained by the nuclear norm and its variants when $X$ is of small rank and small size.
    With the increase of the rank and the size, the SNRs obtained by the rank function become smaller than those obtained by other functions.
    When comparing the parameter strategy of ``DP'' (i.e. ours) and   ``SURE'' {in NN model},
    we observe that the SNRs  are approximately equal, while the SNR obtained by our ``DP'' is slightly larger than that by ``SURE'' in {many of the} cases.
    We also observe that the SNRs obtained by TNN and GWNN are larger than those obtained by NN,  which indicates their superiority in characterizing low-rank properties.

    {In terms of CPU running time, our ``DP'' method  is  almost as fast as the ``HardT'' method as we only need one substitution to
    get $\lambda$ (see (\ref{lambda2}) and (\ref{wlambda})) and then obtain the solution by the thresholding operators similar to the ``HardT" method. In contrast, ``SURE" method requires a minimization of the SURE function (\ref{SURE}) by trial-and-error. According to the CPU running time, our ``DP'' method is at least 11  times faster than the ``SURE''  method.}

    Now we compare the performance of different methods for the randomized algorithms, and we list the SNRs in Table \ref{Table2}.
    Similar conclusions can be observed among different methods.
    {Note that for the nuclear-norm problem, our ``DP'' approach out-perform the ``SURE'' approach in 38 out of the 40 different cases we tested.}
    We also observe that in most cases, the SNRs obtained by the randomized algorithm are equal to or slightly lower than these by the standard algorithm.
    We remark that the difference of the SNRs comes from the difference between   $YQ_\ell Q_\ell^T$ and $ Y$.
    {The \tr{speed} advantage of the randomization algorithm is more obvious with the increase of matrix dimension. When  $m=n=4000$, the random algorithm is 28 times faster than the standard algorithm.}

    \subsection{Simulations on PINCAT Numerical Phantom}
    \begin{table}[t]
    \centering
    \caption{SNRs comparison of different methods for the PINCAT Numerical Phantom.}\Large
    \scalebox{0.55}{
    \begin{tabular}{M{3.2cm} M{1.8cm} M{1.8cm} M{1.8cm} M{1.8cm} M{1.8cm} M{1.8cm} M{1.8cm} M{1.8cm}}
    \hline
    $\tau$ & {Noisy} & {HardT} & {NN-SURE} & {NN-DP}  & {TNN-SURE} & {TNN-DP}&  {GWNN-SURE} & {GWNN-DP}\\
    \hline
    5     & 27.50  & 33.19  & \textit{33.76}  & 33.71  & {34.34}  & \textit{34.42}  & \textit{\textbf{35.03} } & 35.02  \\
    10    & 21.48  & 27.60  & \textit{28.79}  & 28.73  & 29.64  & \textit{29.69}  & \textit{\textbf{30.13 }} & \textit{\textbf{30.13} } \\
    15    & 17.96  & 24.33  &\textit{ 25.90}  & 25.85  & 26.89  & \textit{27.01}  & \textit{\textbf{27.33} } & 27.31  \\
    20    & 15.46  & 22.01  & \textit{23.88}  & 23.83  & 25.09  & \textit{25.14}  & \textit{\textbf{25.34 }} & 25.31  \\
    25    & 13.52  & 20.21  & 22.23  & \textit{22.26}  & \textit{23.73}  & 23.70  & \textit{\textbf{23.81 }} & 23.76  \\
    30    & 11.94  & 18.74  & \textit{21.02}  & {20.98}  & \textit{\textbf{22.54 }} & 22.53  & \textit{\textbf{22.54 }} & 22.48  \\
    35    & 10.60  & 17.49  & 19.84  & \textit{19.90}  & \textit{\textbf{21.59} } & 21.55  & \textit{21.46}  & 21.40  \\
    40    & 9.44  & 16.41  & \textit{18.97}  & 18.96  & \textit{\textbf{20.77 }} & 20.70  & \textit{20.52}  & 20.46  \\
    45    & 8.42  & 15.45  & 18.11  & \textit{18.13}  & \textit{\textbf{19.99 }} & 19.94  & \textit{19.68}  & 19.62  \\
    50    & 7.50  & 14.59  & 17.29  & \textit{17.39}  & \textit{\textbf{19.34 }} & 19.27  & \textit{18.92}  & 18.88  \\
    Average & 14.38  & 21.00  & \textit{22.98}  & 22.97  & {24.39 } & \textit{24.40}  & \textit{\textbf{24.48}}  & 24.44  \\
    \hline
    {Average Time} &   --  & {19.01}  & {733.34}  & {20.12}  & {761.91}  & {20.06}  & {754.24}  & {20.81}  \\ \hline
    \end{tabular}}%
    \label{Table3}
    \end{table}%

    \begin{figure}[t]
    \centering
    \subfigure[Truth]{
    \begin{minipage}[t]{0.22\linewidth}
    \centering
    \includegraphics[height=2.6cm,width=2.6cm]{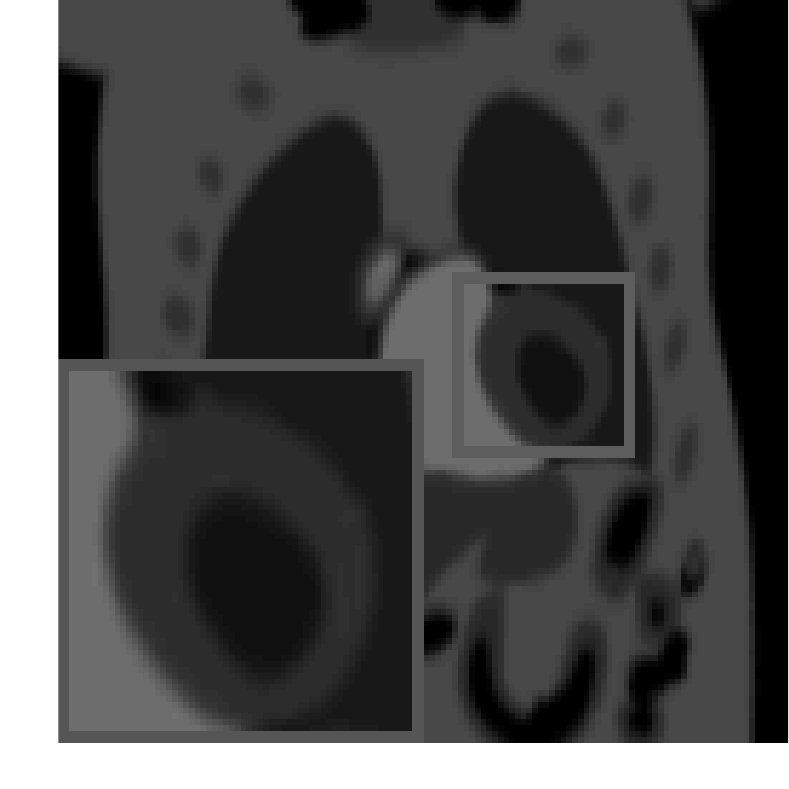}
    \end{minipage}
    }
    \subfigure[Noisy]{
    \begin{minipage}[t]{0.22\linewidth}
    \centering
    \includegraphics[height=2.6cm,width=2.6cm]{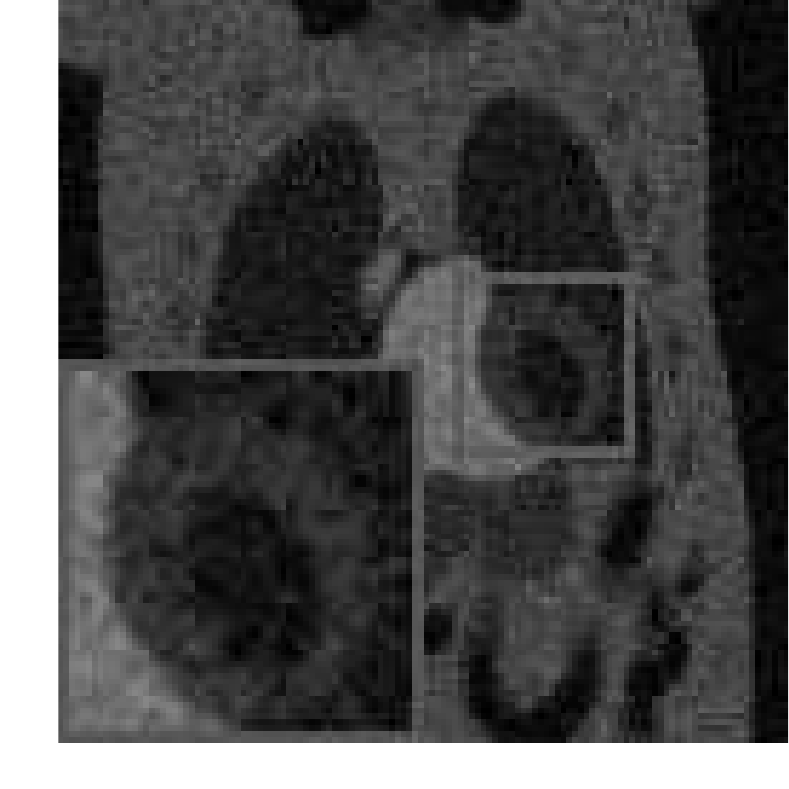}
    \end{minipage}
    }
    \subfigure[HardT]{
    \begin{minipage}[t]{0.22\linewidth}
    \centering
    \includegraphics[height=2.6cm,width=2.6cm]{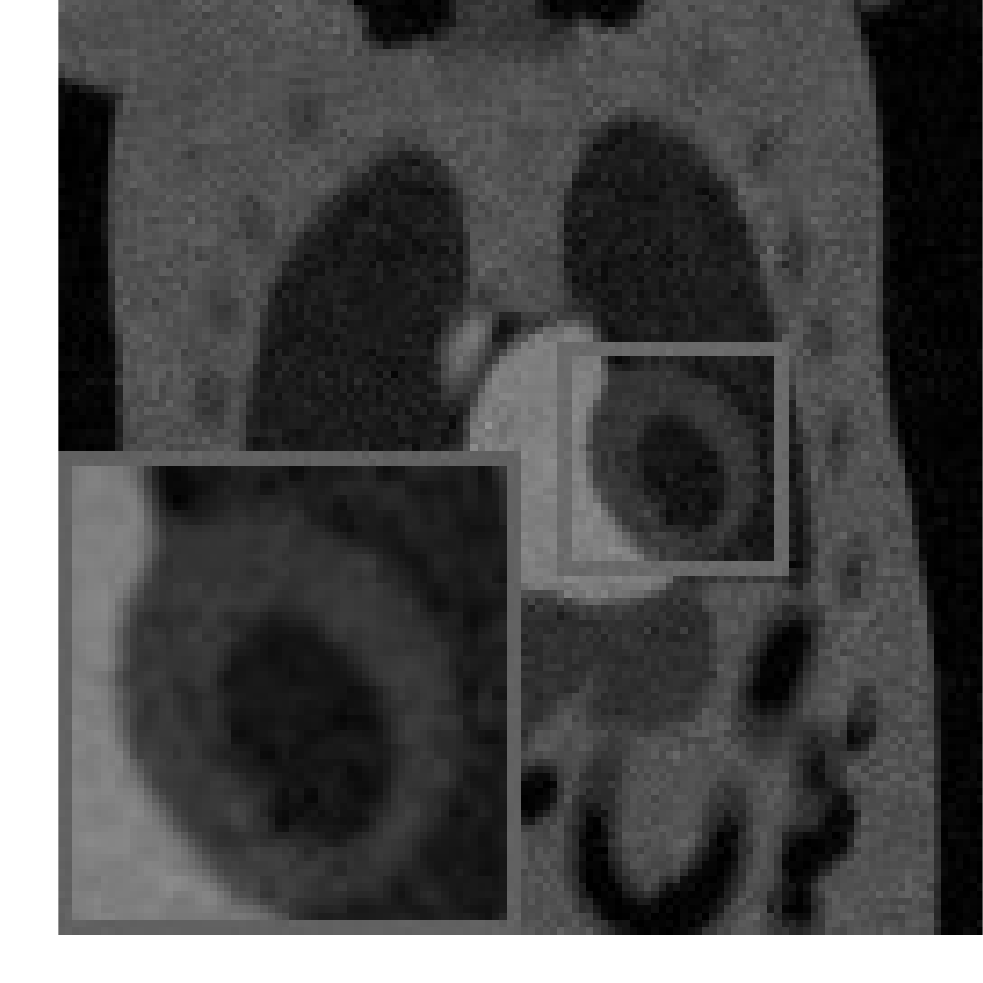}
    \end{minipage}
    }

    \subfigure[NN-SURE]{
    \begin{minipage}[t]{0.22\linewidth}
    \centering
    \includegraphics[height=2.6cm,width=2.6cm]{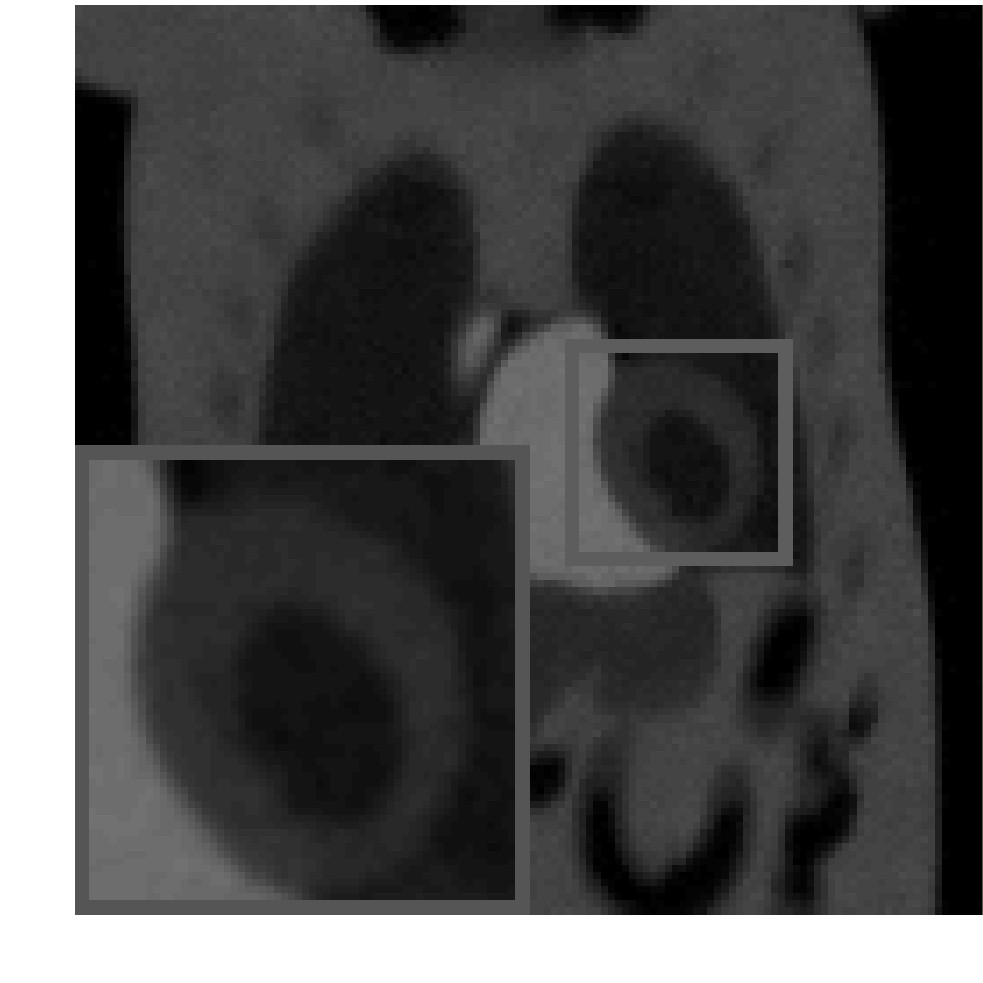}
    \end{minipage}
    }
    \subfigure[TNN-SURE]{
    \begin{minipage}[t]{0.22\linewidth}
    \centering
    \includegraphics[height=2.6cm,width=2.6cm]{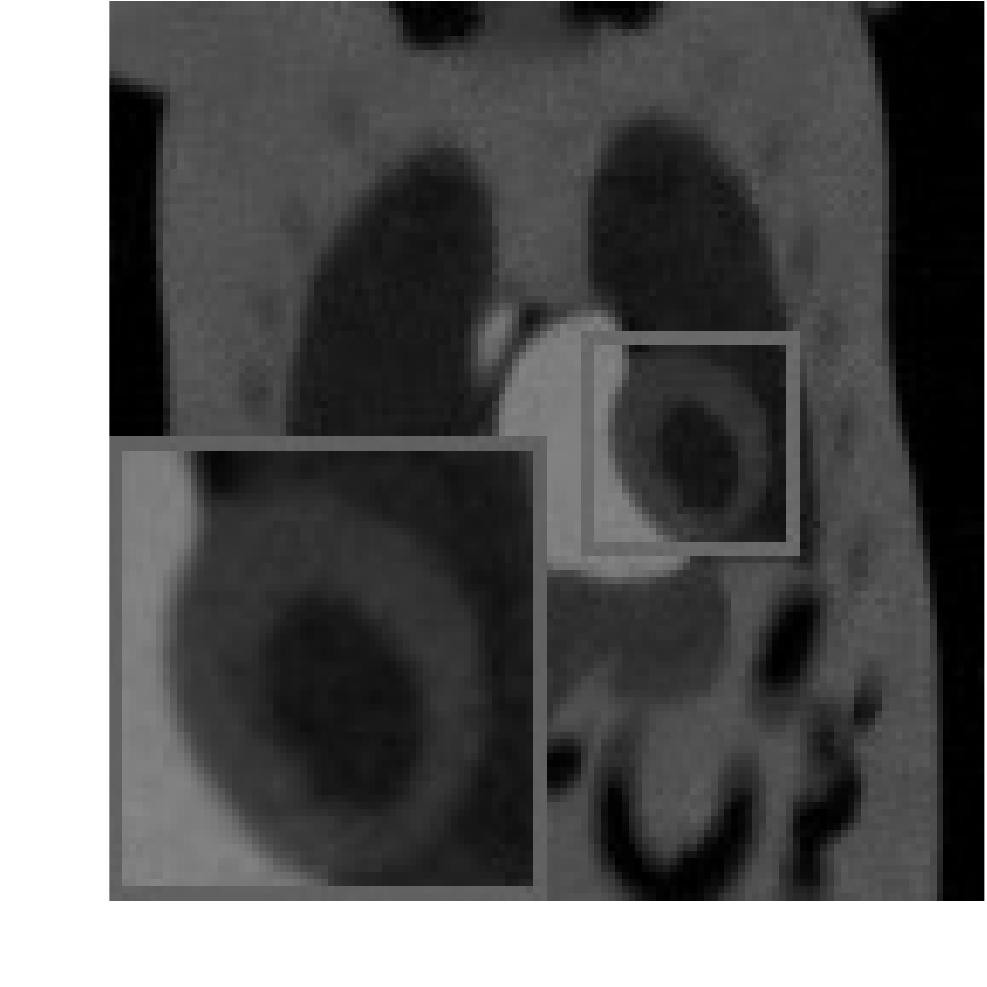}
    \end{minipage}
    }
    \subfigure[WNNM-SURE]{
    \begin{minipage}[t]{0.22\linewidth}
    \centering
    \includegraphics[height=2.6cm,width=2.6cm]{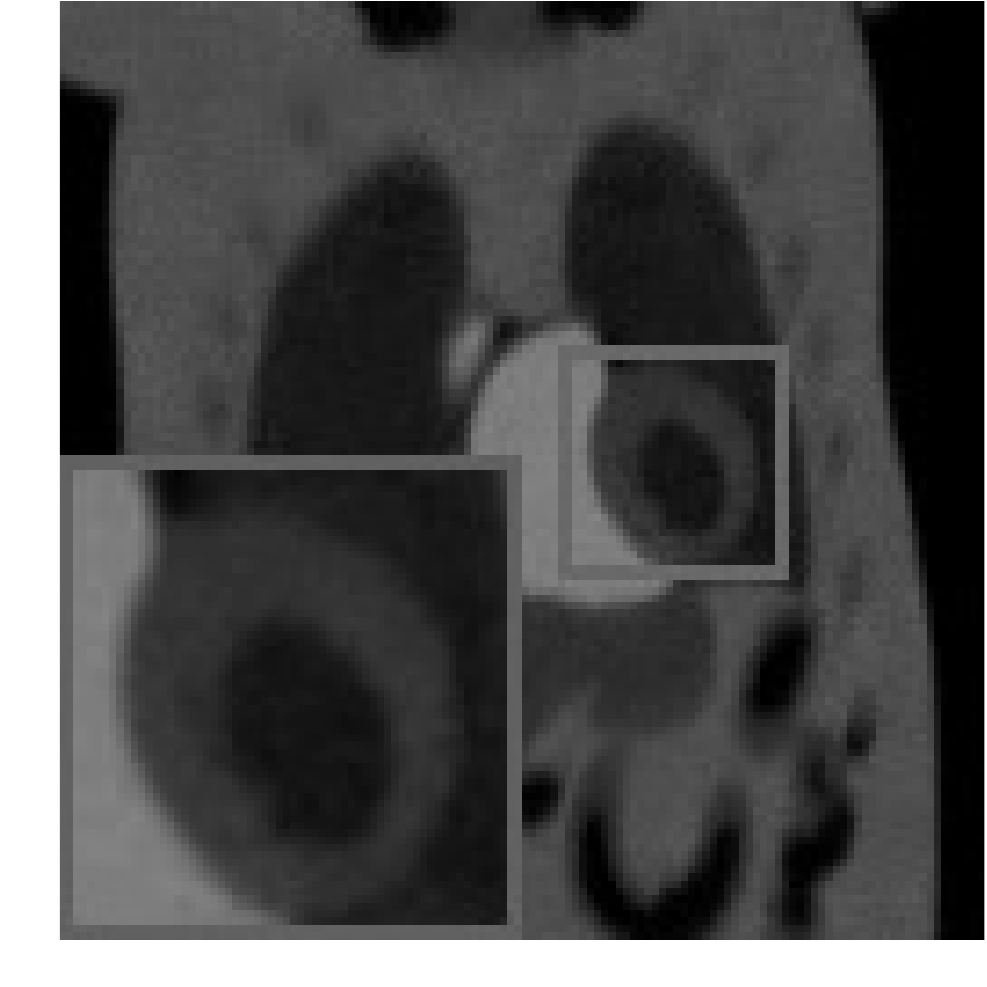}
    \end{minipage}
    }

    \subfigure[NN-DP]{
    \begin{minipage}[t]{0.22\linewidth}
    \centering
    \includegraphics[height=2.6cm,width=2.6cm]{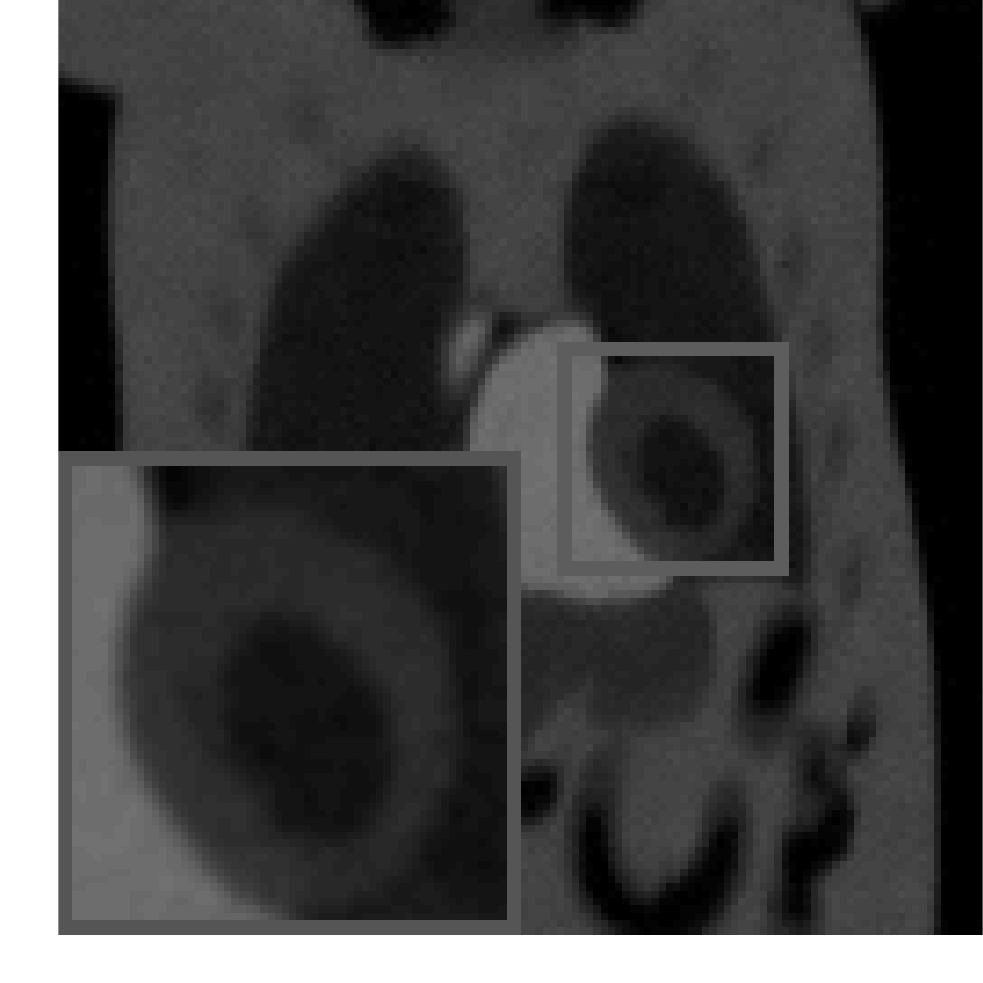}
    \end{minipage}
    }%
    \subfigure[TNN-DP]{
    \begin{minipage}[t]{0.22\linewidth}
    \centering
    \includegraphics[height=2.6cm,width=2.6cm]{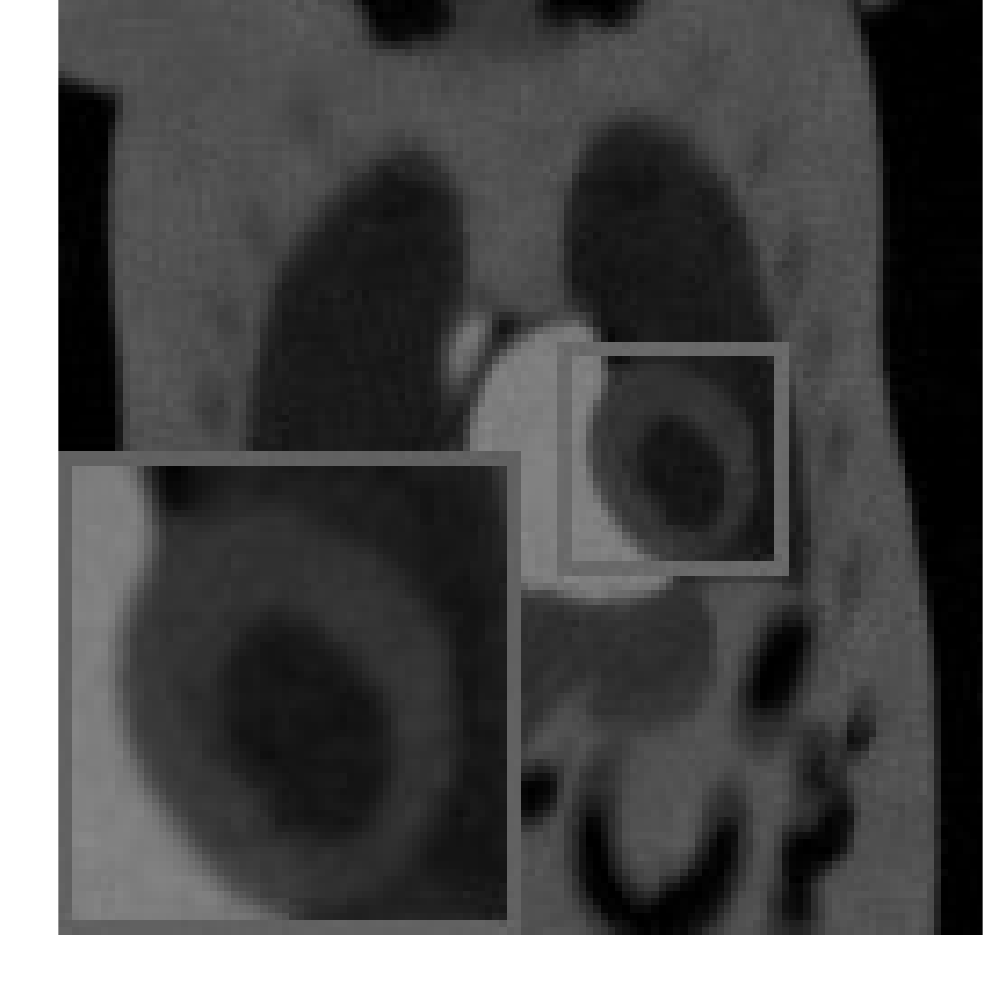}
    \end{minipage}
    }
    \subfigure[GWNN-DP]{
    \begin{minipage}[t]{0.22\linewidth}
    \centering
    \includegraphics[height=2.6cm,width=2.6cm]{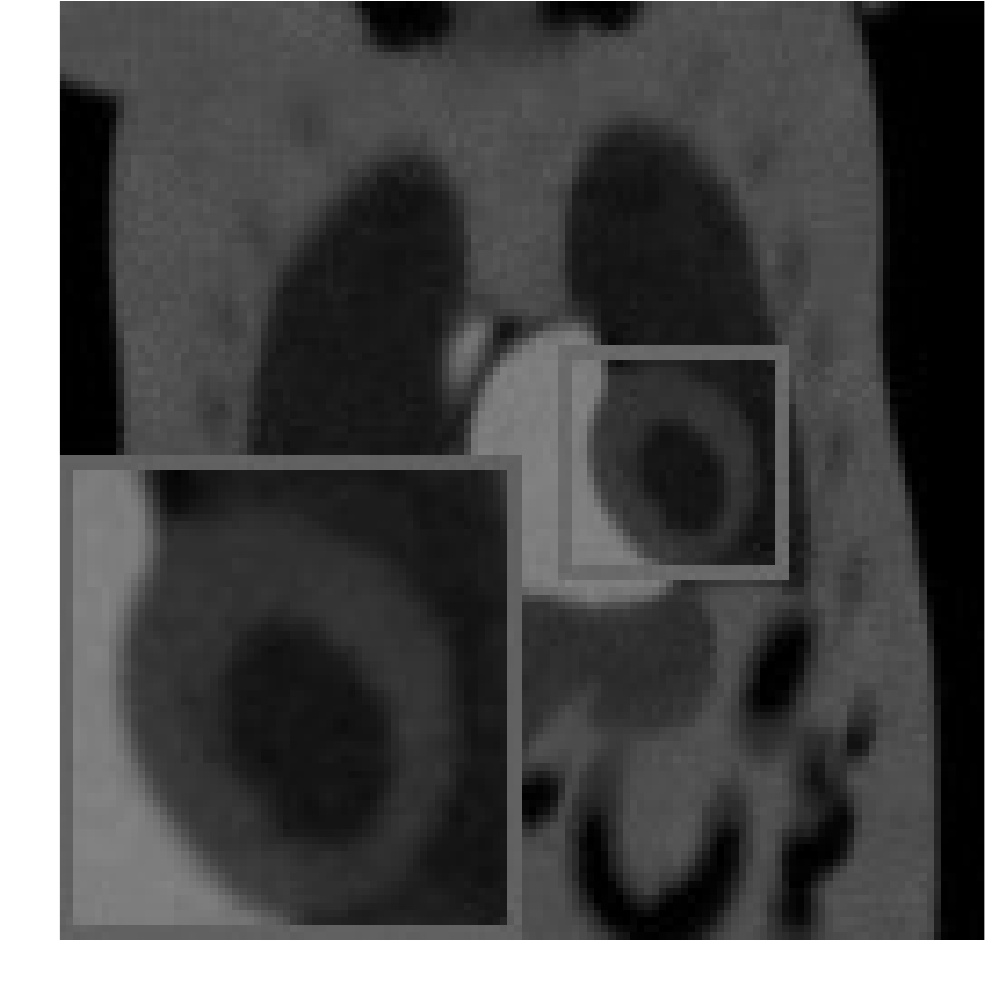}
    \end{minipage}
    }

    \centering
    \caption{
    The truth image with an enlarged portion cropped out from the image for $T=10$, its noisy image by $\tau=30$, and the denoised images by different methods. }
    \label{Local1_figure}
    \end{figure}

    \begin{figure}[h]
    \centering
    \subfigure[NN-SURE]{
    \begin{minipage}[t]{0.27\linewidth}
    \centering
    \includegraphics[height=2.8cm,width=3.5cm]{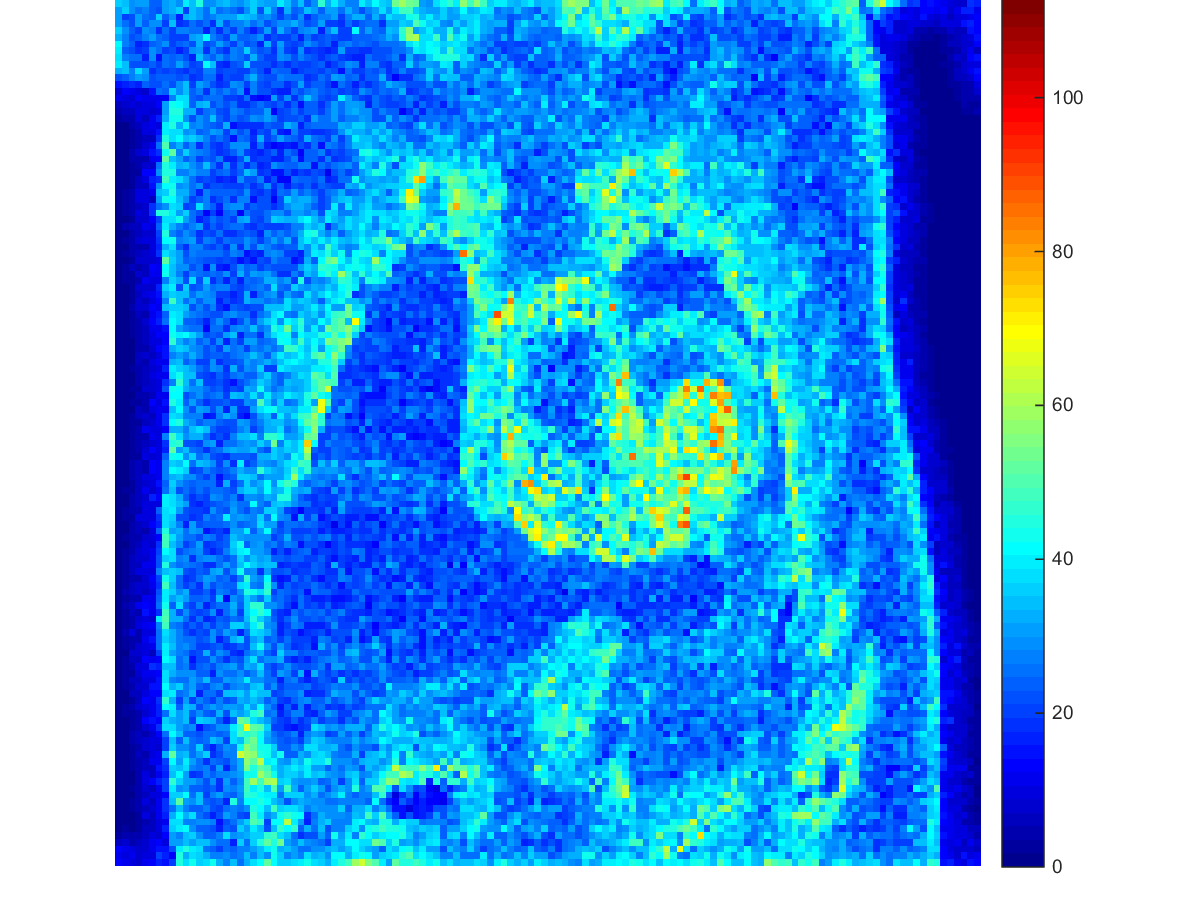}
    \end{minipage}
    }
    \subfigure[TNN-SURE]{
    \begin{minipage}[t]{0.27\linewidth}
    \centering
    \includegraphics[height=2.8cm,width=3.5cm]{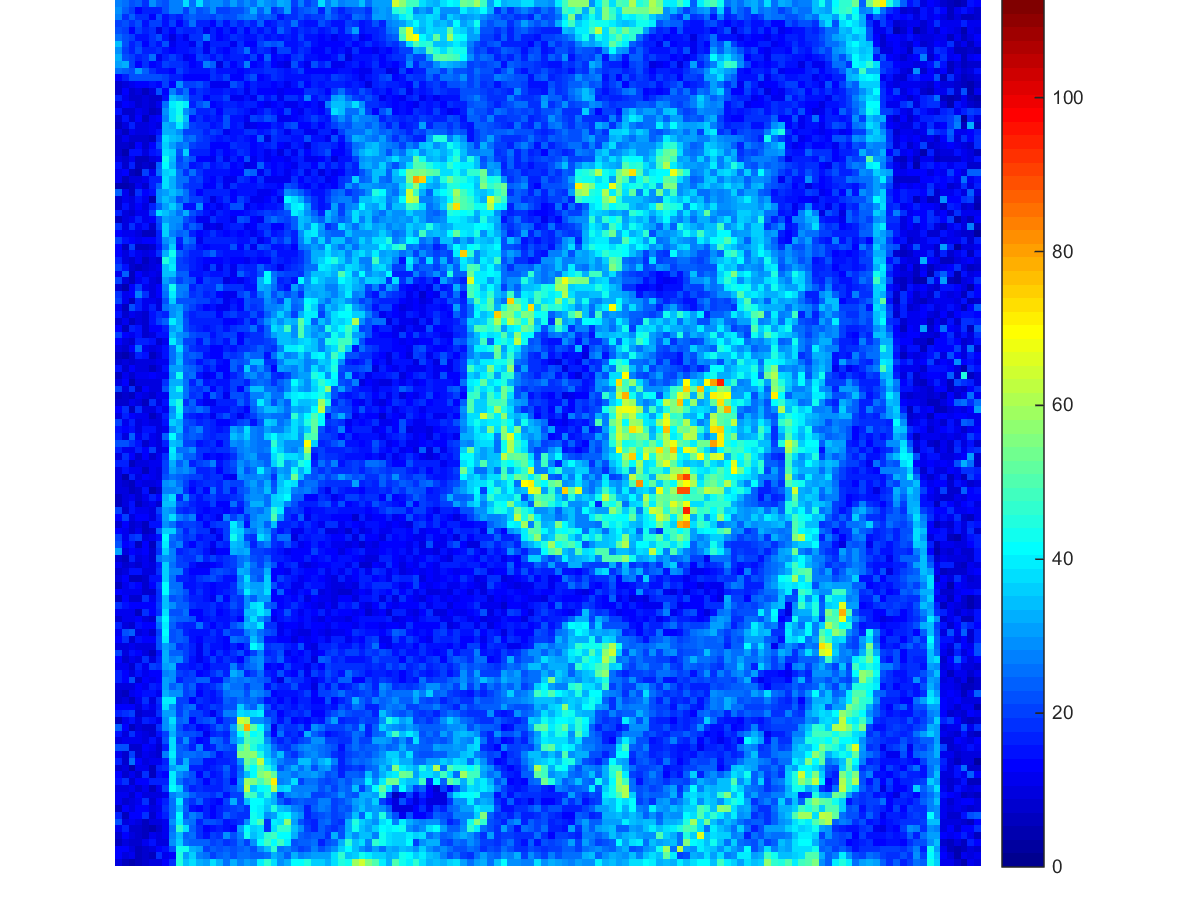}
    \end{minipage}
    }
    \subfigure[GWNN-SURE]{
    \begin{minipage}[t]{0.27\linewidth}
    \centering
    \includegraphics[height=2.8cm,width=3.5cm]{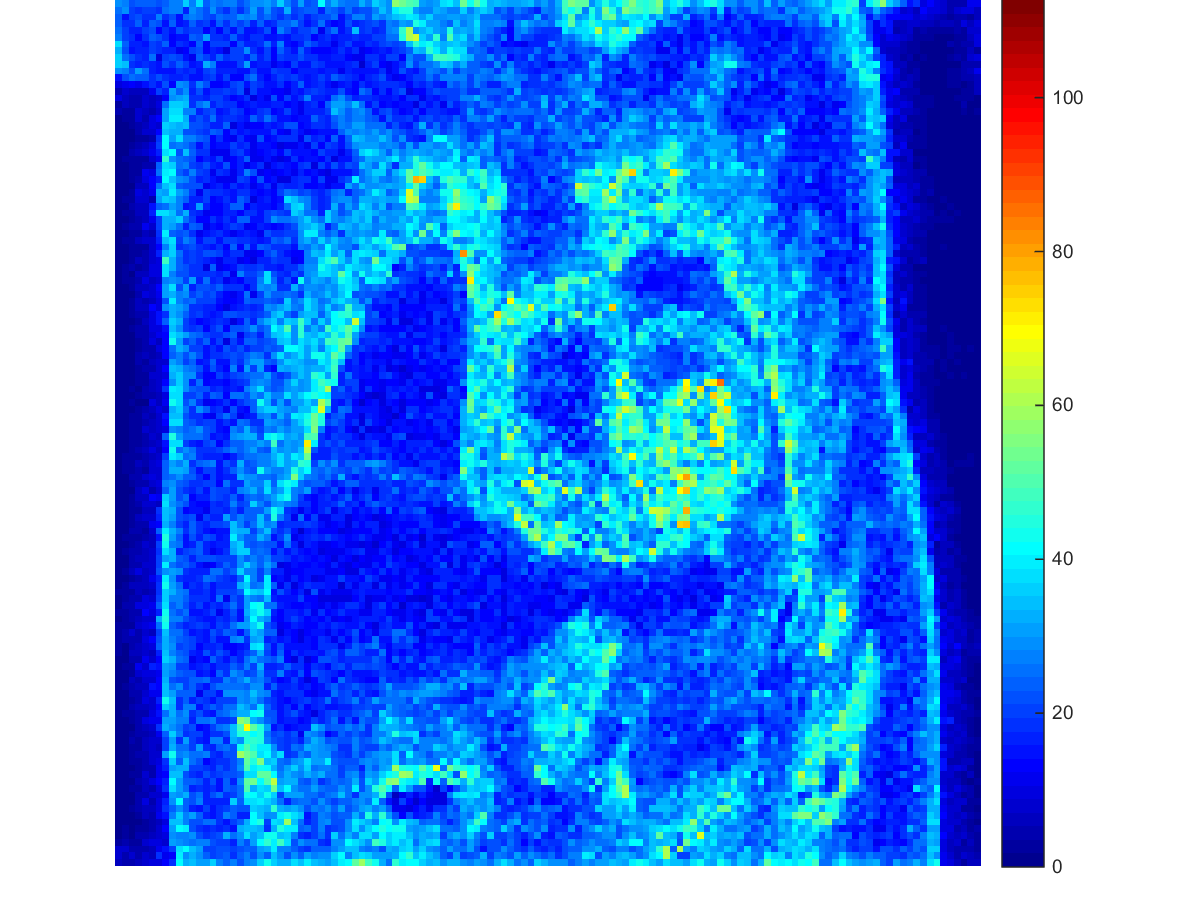}
    \end{minipage}
    }

    \subfigure[NN-DP]{
    \begin{minipage}[t]{0.27\linewidth}
    \centering
    \includegraphics[height=2.8cm,width=3.5cm]{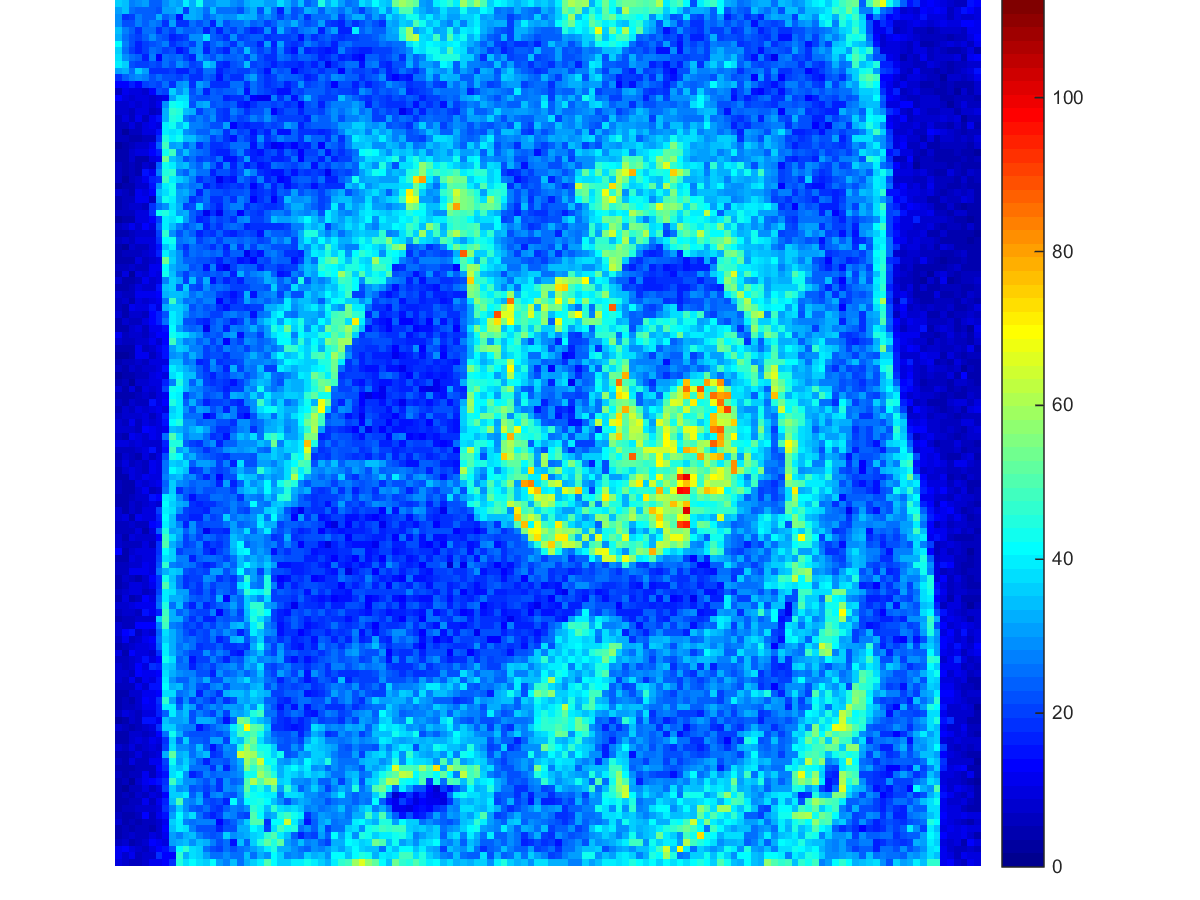}
    \end{minipage}
    }
    \subfigure[TNN-DP]{
    \begin{minipage}[t]{0.27\linewidth}
    \centering
    \includegraphics[height=2.8cm,width=3.5cm]{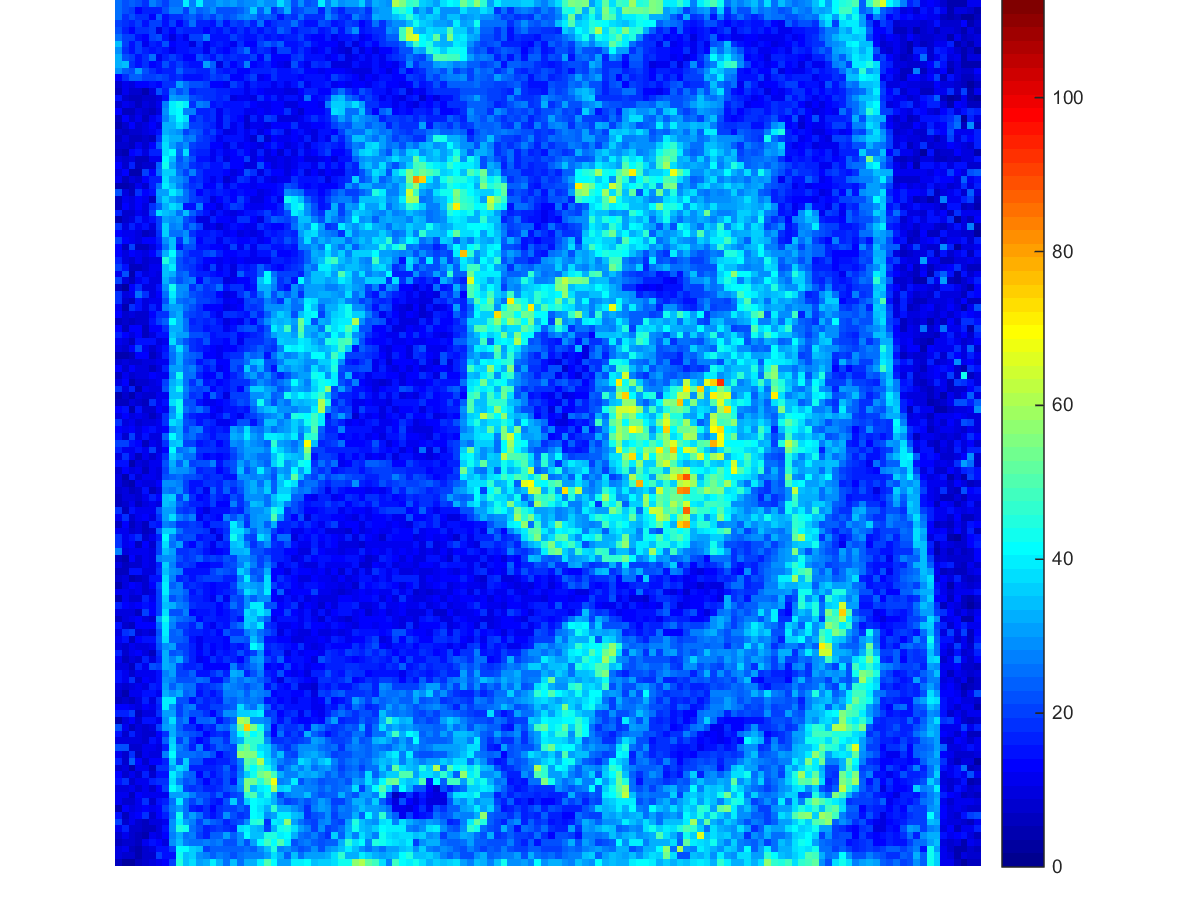}
    \end{minipage}
    }
    \subfigure[GWNN-DP]{
    \begin{minipage}[t]{0.27\linewidth}
    \centering
    \includegraphics[height=2.8cm,width=3.5cm]{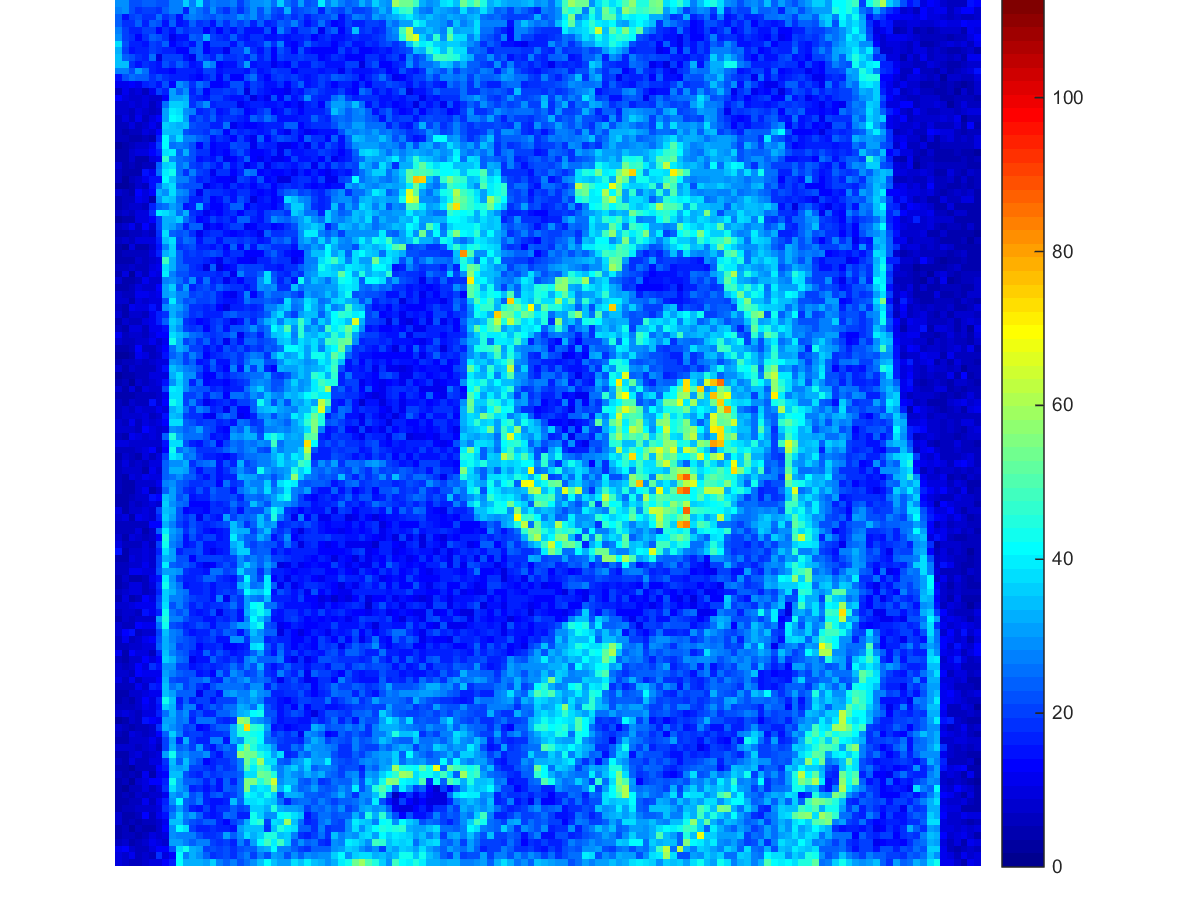}
    \end{minipage}
    }
    \centering
    \caption{The worst-case error through time of different methods for the the PINCAT Numerical Phantom.}
    \label{error1}
    \end{figure}

    \par In the following, we compare the performance of different methods in dynamic MRI-denoising. We use the same data set and parameter settings presented in \cite{CandesUnbiased}. There are 50 images of size $128 \times 128$ taken at 50 time-steps.
    The physiologically-improved NCAT (PINCAT) numerical phantom \cite{sharif2007physiologically} simulates a first-pass myocardial perfusion real-time MRI series.
    The free-breathing model is available in the kt-SLR software package \cite{lingala2011accelerated}.
    The noise is mainly caused by  thermal noise and physiological noise, which follows the Gaussian distribution.
    Complex identical and independent  Gaussian noise were added to the image data.

    The locally low-rank recovery (LLR) method  \cite{trzasko2011local}  is applied to remove the noise.
    \tr{LLR is the generalization of low-rank matrix  reconstruction, which spatially partitions the image sequence into small blocks by a fixed sliding window with stride 1. In the experiment, the sliding window size is set to $7\times 7$. We reshape each partitioned block into a $49 \times 50$ Casorati matrix (a matrix whose columns comprise vectorized patch of the image sequence).  Because the changes of each frame in the sequence {are} small, the Casoratic matrix constructed from a clean image sequence is {of} low-rank.
    We solve the low-rank matrix minimization problem by using each of these Casoratic matrices as an observed matrix, and there are $128^2=16,384$ of such minimization problems to solve.}

    The SNR value is used to evaluate quantitatively the performance of different methods, which are listed in  Table \ref{Table3}. \tr{The best results among all methods are shown in boldface and  the best results for the given regularization norm are marked in italic font.} {We see from the average SNR values that our method is only 0.01dB lower than SURE for ``NN" and ``TNN" and 0.04dB lower for ``GWNN".} {However, since we have to perform denoising on \tr{$16,384$}  minimization problems of size $49 \times 50$ each, the ``SURE'' method is very time-consuming. Our ``DP'' method is at least 35 times faster than the ``SURE'' method.}

    Figure \ref{Local1_figure} shows the truth image with an enlarged portion cropped out from the image at the 10th time-step, the noisy image with $\tau=30$, and the denoised images by different methods,  respectively.
    We observe that there is still residual noise in the denoised image obtained by ``HardT", which implies that the parameter chosen {as in \cite{gavish2014the}} is not suitable for LLR. This is because the threshold selected  is related to the noise level only, and therefore the same threshold is selected for each block.

    We also compare the  worst-case absolute error\footnote{The  worst-case absolute error is computed by $\max_t |X_{ij}(t)-\widehat{X}_{ij}(t)|$, where $X(t)$ and $\widehat{X}(t)$ are the truth image and the recovered image at time $t$.}  for our ``DP'' and ``SURE'' methods, which is shown in  Figure \ref{error1}.
    From the figures, we observe that {``NN''} exhibits higher residual error than {``TNN''} and ``GWNN''.
    Once again, the figures verify that the results obtained by the truncated nuclear  norm or weighted nuclear norm  are better than those by the standard nuclear norm or the hard thresholding scheme.

    %

    \section{Conclusion}\label{Conclusion}
    The constrained model and the regularized model for the low-rank matrix recovery {were} considered in this paper.
    We have derived a formula for the regularization parameter when a bound of the residual norm is given.
    The results {were} used to select the regularization parameter automatically using the discrepancy principle and to solve the constrained problem when the bound of the constraint is given.
    Experimental results {showed} that the proposed approach is competitive to the SURE methods in terms of noise removal and adaptive parameter selection, and is much faster in terms of CPU time.



    \bibliographystyle{plain}

\end{document}